\documentclass[11pt,reqno]{amsart}
\usepackage[top=1in, bottom=1in, left=1in, right=1in]{geometry}
\usepackage{times, amsthm, amssymb, amsmath, amsfonts, amsthm, bm, graphicx, mathrsfs,mathtools}
\usepackage[all]{xy}
\usepackage[usenames,dvipsnames]{color}
\usepackage[colorlinks=true,linkcolor=blue,urlcolor=blue, citecolor=blue]%
  {hyperref}
\usepackage[alphabetic]{amsrefs}
\usepackage{array}

\usepackage{tikz}
\usepackage{tikz-cd}
\usetikzlibrary{patterns,calc,arrows,shapes}

\newcommand{\excise}[1]{}

\newtheorem{theorem}{Theorem}[section]
\newtheorem*{thm}{Theorem}
\newtheorem{lemma}[theorem]{Lemma}

\newtheorem{cor}[theorem]{Corollary}
\newtheorem{corollary}[theorem]{Corollary}
\newtheorem{prop}[theorem]{Proposition}
\newtheorem*{Propo}{Proposition}
\newtheorem*{definition*}{Definition}

\newtheorem{theorem/definition}[theorem]{Theorem/Definition}

\theoremstyle{definition}
\newtheorem{example}[theorem]{Example}
\newtheorem{remark}[theorem]{Remark}

\newtheorem{defn}[theorem]{Definition}
\newtheorem{definition}[theorem]{Definition}

\def\<{\langle}
\def\>{\rangle}

\renewcommand{\gets}{\longleftarrow}
\renewcommand{\geq}{\geqslant}

\newcommand{\PP}{\ensuremath{\mathbb{P}}}

\newcommand{\cB}{\ensuremath{\mathcal{B}}}

\newcommand{\cX}{\ensuremath{\mathcal{X}}}

\DeclareMathOperator{\codim}{codim}

\DeclareMathOperator{\colo}{color}
\DeclareMathOperator{\cone}{cone}\DeclareMathOperator{\Cox}{Cox}
\DeclareMathOperator{\Ext}{Ext}
\DeclareMathOperator{\Ann}{Ann}

\DeclareMathOperator{\Hom}{Hom}

\DeclareMathOperator{\link}{lk}
\DeclareMathOperator{\pdim}{pdim}
\DeclareMathOperator{\Pic}{Pic}

\DeclareMathOperator{\Spec}{Spec}

\DeclareMathOperator{\Tor}{Tor}

\DeclareMathOperator{\vdim}{vdim}
\DeclareMathOperator{\mathcolor}{color}

\DeclareMathOperator{\join}{join}

\numberwithin{equation}{section}
\begin{document}

\vspace*{-12mm}
\mbox{}
\title[Homological and combinatorial aspects of virtually Cohen--Macaulay sheaves]{Homological and combinatorial aspects of \\ virtually Cohen--Macaulay sheaves}

\author{Christine Berkesch}
\address{School of Mathematics \\
University of Minnesota.}
\email{cberkesc@umn.edu}

\author{Patricia Klein}
\address{School of Mathematics \\
University of Minnesota.}
\email{klein847@umn.edu}

\author{Michael C. Loper}
\address{Department of Mathematics \\
University of Wisconsin - River Falls.}
\email{michael.loper@uwrf.edu}

\author{Jay Yang}
\address{School of Mathematics \\
University of Minnesota.}
\email{jkyang@umn.edu}

\thanks{CB was partially supported by NSF Grants DMS 1661962 and 2001101. \\
\indent JY was partially supported by NSF RTG Grant 1745638. 
}

\begin{abstract} 
When studying a graded module $M$ over the Cox ring of a smooth projective toric variety $X$, there are two standard types of resolutions commonly used to glean information: free resolutions of $M$ and vector bundle resolutions of its sheafification.  Each approach comes with its own challenges.   There is geometric information that free resolutions fail to encode, while vector bundle resolutions can resist study using algebraic and combinatorial techniques.  Recently, Berkesch, Erman, and Smith introduced virtual resolutions, which capture desirable geometric information and are also amenable to algebraic and combinatorial study. The theory of virtual resolutions includes a notion of a virtually Cohen--Macaulay property, though tools for assessing which modules are virtually Cohen--Macaulay have only recently started to be developed.  

In this paper, we continue this research program in two related ways.   The first is that, when $X$ is a product of projective spaces, we produce a large new class of virtually Cohen--Macaulay Stanley--Reisner rings, which we show to be virtually Cohen--Macaulay via explicit constructions of appropriate virtual resolutions reflecting the underlying combinatorial structure.  The second is that, for an arbitrary smooth projective toric variety $X$, we develop homological tools for assessing the virtual Cohen--Macaulay property.  Some of these tools give exclusionary criteria, and others are constructive methods for producing suitably short virtual resolutions.  We also use these tools to establish relationships among the arithmetically, geometrically, and virtually Cohen--Macaulay properties.
\end{abstract}
\maketitle

\mbox{}
\vspace*{-16mm}

\setcounter{section}{1}
\section*{Introduction}
\label{sec:intro}

Let $X$ be a smooth projective toric variety over an algebraically closed field $k$ with Cox ring $S$ and irrelevant ideal $B$ (see \cite[\S5.2]{cox-little-schenck}). 
The Cox ring $S$ is a positively $\operatorname{Pic}(X)$-graded polynomial ring,  i.e., $\operatorname{Pic}(X) \cong \mathbb{Z}^d$, and the multidegrees of the variables of $S$ lie in a single open half-space of  $\mathbb{Z}^d$.
There is a correspondence between $\operatorname{Pic}(X)$-graded $B$-saturated modules $M$ over $S$ and sheaves $\widetilde{M}$ on $X$~\cite{audin,musson,cox:homog} (see \cite{mustata-toric} when $X$ is not smooth). 
Unfortunately, the numerics of the minimal $\operatorname{Pic}(X)$-graded free resolutions for such $S$-modules do not obviously provide many geometric insights for $M$ when $X$ is not projective space. 
For example, a minimal $\operatorname{Pic}(X)$-graded free resolution of $M$ may be significantly longer than the dimension of $X$. 
However, this failure appears to be a consequence of imposing too much algebraic structure on the resolution. 
Approaching the problem from the geometric perspective, vector bundle resolutions of $\widetilde{M}$ are bounded in length by the dimension of $X$, but vector bundles on $X$ are significantly more complicated than line bundles on $X$.  
A proposed solution comes from \cite{virtual-original}, in which the authors introduce a type of resolution of $M$ by free $S$-modules, which they call a \emph{virtual resolution}, that better captures geometrically meaningful properties of $\operatorname{Pic}(X)$-graded $S$-modules, such as unmixedness, well-behavedness of deformation theory, and regularity of tensor products.  
Because virtual resolutions are defined up to the sheafification of $M$, the object of study is intrinsically geometric. 
Because the resolutions themselves are in the category of $S$-modules, they are naturally amenable to algebraic techniques.  

Although virtual resolutions are desirable for their ability to encode geometric information, we do not yet have the wealth of tools for studying them that we do for studying graded free resolutions.  In particular, there are not yet many methods for constructing short virtual resolutions or for establishing the minimum possible length among the virtual resolutions of a chosen $\operatorname{Pic}(X)$-graded $S$-module $M$.  We provide some of each in this paper.

Our broad goal in this article is to work towards a rich understanding of \emph{virtually Cohen--Macaulay} modules (or virtually Cohen--Macaulay coherent sheaves) as an analogue to Cohen--Macaulay modules over the coordinate rings of affine or projective space.  
We provide two methods to construct short virtual resolutions either from longer virtual resolutions or from short resolutions of closely-related modules. These methods can be helpful in establishing that modules are virtually Cohen--Macaulay (see Propositions~\ref{prop:mapping-cone} and~\ref{prop:vreg-elt-quotient}).  
We also obtain homological obstructions to being virtually Cohen--Macaulay (see Section~\ref{sec:virtual-ext}).  Guiding these structural developments is our production of a large class of virtually Cohen--Macaulay Stanley--Reisner rings in Section~\ref{sec:triangles}.  
The results on this class are hard won through the careful application of Hochster's formula, interpreted in a virtual setting, together with an analysis of the spectral sequence associated to a certain nerve complex.  
This not only provides us with a new source of examples of virtually Cohen--Macaulay modules as we work to develop the theory, but also, given the difficulty of studying even Stanley--Reisner rings in this context, highlights the need for the advent of more virtual homological tools.

\subsection*{Acknowledgements}
\label{subsec:acknowledgements}

We would like to thank Daniel Erman and Gregory G. Smith for helpful conversations related to this work.  We are also grateful to the anonymous referee for valuable feedback on a previous version of this paper.

\section{Background and Statements of Main Results}
\label{sec:prelim}

Throughout this article, let $X$ be a smooth projective toric variety over the algebraically closed field $k$, and let $S = \Cox(X)$ with irrelevant ideal $B$.  All $S$-modules are assumed to be finitely generated and $\Pic(X)$-graded, and all sheaves are assumed to be coherent.  

Let $M$ be an $S$-module.  As in~\cite[Definition~1.1]{virtual-original}, a free graded $S$-complex $F_\bullet=[F_0\gets F_1\gets\dotsb]$ is a \emph{virtual resolution} of $M$ (or of $\widetilde{M}$) if the corresponding complex $\widetilde{F}_\bullet$ of vector bundles is a locally free resolution of the sheaf $\widetilde{M}$. 
Next, define the \emph{virtual dimension} of $M$, denoted $\vdim M$, 
to be the minimal length of a virtual resolution of $M$.  For products of projective spaces, there is an inequality $\vdim M\ge\codim M$ (\cite[Proposition~2.5]{virtual-original}); in light of this and an analogue to the affine case, we say that $M$ is \emph{virtually Cohen--Macaulay} if $\widetilde{M} \neq 0$ and $\vdim M = \codim M$, the minimum possible.  
We say that a subscheme $V\subset X$ is \emph{virtually Cohen--Macaulay} if its Cox ring is virtually Cohen--Macaulay as an $S$-module.  Although there is a precise description in the literature for when complexes are virtual resolutions (see \cite{Lop}), little is known about how to 
assess the virtual dimension of a module 
or how to  construct virtual resolutions of minimal length, even when that minimal length is known.

In this paper, we construct virtual resolutions of minimal length for a family of Stanley--Reisner rings in order to show that they are virtually Cohen--Macaulay.  Before stating that result, we review the \emph{Stanley--Reisner correspondence} between simplicial complexes and squarefree monomial ideals.  For a detailed introduction, we refer the reader to \cite{MS04}. 

\begin{definition*}
Let $\Delta$ be a simplicial complex on $\{1,2,\dots,n\}$ and $R = k[x_1,x_2, \ldots, x_n]$.  Define the \emph{Stanley--Reisner ideal} of $\Delta$ to be 
\[
I_\Delta = \left\< x_{i_1}x_{i_2} \cdots x_{i_k} \mid \{i_1, \ldots, i_k\} \notin \Delta \right\>
\] 
and the \emph{Stanley--Reisner ring} of $\Delta$ to be $R/I_\Delta$.
\end{definition*}

We now state our main result on the existence of a new family of virtually Cohen--Macaulay rings (see Theorem~\ref{thm:monomial-vcm}). 

\begin{thm}
Let $S$ be the Cox ring of $X = \PP^{n_1} \times \PP^{n_2}\times \cdots \times \PP^{n_r}$. 
If $\Delta$ is an $r$-dimensional simplicial complex and the variety $V(I_\Delta)\subseteq X$ is equidimensional, 
then $S/I_{\Delta}$ is virtually Cohen--Macaulay. 
\end{thm}

Relationships between $\vdim M$ and $\dim X$ have been of interest since the introduction of virtual resolutions.
In~\cite[Proposition~1.2, Theorem 5.1]{virtual-original}
a Hilbert Syzygy Theorem-type bound, $\vdim M \le \dim X$, was given for an arbitrary $\Pic(X)$-graded $S$-module $M$ when $X$ is a product of projective spaces and for an arbitrary punctual scheme in any smooth projective toric variety $X$. 
Further, \cite{yang-monomial} shows that $\vdim S/I \le \dim X$ when $I$ is a relevant (i.e., $B^t \not\subseteq I$ for all $t \geq 1$) monomial ideal of $S$ and $X$ is a smooth projective toric variety.  Our new result most directly compares with a similar theorem in the case of pure and balanced simplicial complexes, which are necessarily of dimension $r-1$ (see \cite[Theorem 1.3]{reu2019}). 
Our proof is constructive, and we illustrate its use in building explicit resolutions in Examples~\ref{ex:disjoint-lines} and~\ref{ex:monomial-nonCM-components}.

Our second construction of short virtual resolutions for the purpose of realizing the virtual Cohen--Macaulay property comes by way of a mapping cone.  It is precisely stated and proved as Proposition~\ref{prop:mapping-cone} and summarized below.   There is not a directly analogous strategy for shortening locally free resolutions over $\mathbb{P}^n$.  As such, this result is an example of a tool that is new to the virtual setting, rather than being a modification of a tool from the projective setting.

\begin{Propo}
Let $F_\bullet$ be a virtual resolution of an $S$-module $M$ of length $t$ such that $\Ext^t(M,S)^\sim = 0$. If $\Ext^t(M,S)$ admits a free resolution of length at most $t+1$, then we can construct a virtual resolution of $M$ of length $t-1$. 
\end{Propo}

Additionally, we propose a notion of a \emph{virtually regular element} (see Definition~\ref{def:vreg-element}) and show (as Proposition~\ref{prop:vreg-elt-quotient}) that it can be used to produce virtual resolutions. 

\begin{Propo}
If $M$ is a virtually Cohen--Macaulay $S$-module and $f$ is a virtually regular element on $M$, then $M/fM$ is virtually Cohen--Macaulay.
\end{Propo}

In Example~\ref{ex:vreg-element}, we use Proposition~\ref{prop:mapping-cone} to show that a particular squarefree monomial ideal defines a virtually Cohen--Macaulay quotient ring.  We then quotient by a sequence of virtually regular elements to arrive at a virtually Cohen--Macaulay quotient ring outside of the squarefree monomial setting.  This virtually Cohen--Macaulay quotient ring has an embedded associated prime, which we notice is irrelevant.

One may note that if $M$ is an arithmetically Cohen--Macaulay $S$-module, then $M$ is virtually Cohen--Macaulay and that, if $M$ is virtually Cohen--Macaulay, then $M$ is \emph{geometrically Cohen--Macaulay} (Definition~\ref{def:geomCM}, Proposition~\ref{prop:aCM=>vCM=>gCM}).  Additionally, if $M$ is virtually Cohen--Macaulay, then, if it has any associated primes of height other than $\codim M$, they must be irrelevant (Proposition~\ref{prop:equidim}), i.e., a virtually Cohen--Macaulay module is virtually unmixed.  Through Examples~\ref{ex:disjoint-lines} and~\ref{ex:tangentbundle}, we show that all of these implications are strict.

In addition to the tools above, we also provide some homological tools that provide exclusionary criteria for a module to have the virtually Cohen--Macaulay property.  
We also note that because the corresponding complex of sheaves is a locally free resolution, any virtual resolution can be used to compute $\Ext$ and $\Tor$ modules up to sheafification (as Propositions~\ref{prop:vExt} and~\ref{prop:vTor}), which implies the following (as Corollaries~\ref{cor:vanishingExt} and~\ref{cor:vanishingTor}). 

\begin{Propo}
If the $S$-module $M$ has a virtual resolution of length $\ell$, then $\Ext^i_S(M,N)^\sim = 0 = \Tor_i^S(M,N)^\sim$ for all $S$-modules $N$ and all $i>\ell$.
\end{Propo}

We  show how this fact can be used to give an example of a module that cannot be virtually Cohen--Macaulay (see Example~\ref{ex:disjoint-lines-P4}).  We also show that the converses to Proposition~\ref{prop:vreg-elt-quotient} and Corollaries~\ref{cor:vanishingExt} and~\ref{cor:vanishingTor} are false.  In particular, we see that the virtual dimension of a module is bounded below by the homological dimension of the associated coherent sheaf. Further, the homological dimension may be strictly greater than the virtual dimension, as demonstrated in Example~\ref{ex:tangentbundle}.   

\section{Virtually Cohen--Macaulay Stanley--Reisner rings}
\label{sec:triangles}

The purpose of this section is to prove the following theorem. 

\begin{theorem}
\label{thm:monomial-vcm}
Let $S$ be the Cox ring of $X = \PP^{n_1} \times \PP^{n_2}\times \cdots \times \PP^{n_r}$. 
If $\Delta$ is an $r$-dimensional simplicial complex and its associated variety $V(I_\Delta)\subseteq X$ is equidimensional, 
then $S/I_{\Delta}$ is virtually Cohen--Macaulay. 
\end{theorem}

With the dimension constraints in Theorem~\ref{thm:monomial-vcm}, the variety $V(I_\Delta)\subseteq X$ must satisfy $\dim (V(I_\Delta))= 1$, and so this is a theorem about one dimensional subvarieties of $X$ determined by monomial ideals. As such, it is natural to ask whether such a statement holds in a more general setting, for example by taking $X$ to be any smooth projective toric variety. The techniques outlined in this section rely heavily on the structure of the Cox ring of a product of projective spaces. In particular, we rely on our ability to separate the variables into groups based on the product structure of $X$.  Thus, though the statement may hold with fewer hypotheses on $X$, the proof would have to be meaningfully different than that given here.

Let $S = k[x_{i,j}\mid {1\le i\le r, 0\le j\le n_i}]$ be the Cox ring of $X$ (of Theorem~\ref{thm:monomial-vcm}) and $B$ the irrelevant ideal of $S$.  Throughout this section, we will consider simplicial complexes on the vertex set $\cX$ corresponding to the variables $(x_{i,j})_{1\le i\le r, 0\le j\le n_i}$ of $S$.  The vertices in $\cX$ corresponding to $x_{i,\bullet}$ are said to have color $i$.  Let $\Delta$ be a simplicial complex with vertices in $\cX$. 
Define the \emph{color set of a face} $\sigma\in\Delta$ to be the set of the colors of the vertices of $\sigma$, denoted by $\colo(\sigma)$.
We say that a face $\sigma\in\Delta$ is \emph{relevant} if $\colo(\sigma)=[r] = \{1,2,\dots,r\}$ and \emph{irrelevant} otherwise. 
A simplicial complex $\Delta$ is \emph{relevant} if it contains at least one relevant face, and it is \emph{irrelevant} otherwise. 
Note that if $\Delta$ is an irrelevant simplicial complex on $\cX$, then $S/I_{\Delta}$ is irrelevant, i.e., the support of $S/I_{\Delta}$ is contained in $V(B) = \{P \in \Spec(S) \mid B \subseteq P\}$. If $\Delta$ is a relevant simplicial complex on $\cX$, then $\Delta$ is said to be \emph{virtually Cohen--Macaulay} if  $S/I_\Delta$ is virtually Cohen--Macaulay. 

Our proof of Theorem~\ref{thm:monomial-vcm} begins with a lemma treating irrelevant faces of a fixed dimension.  We aim to understand reduced simplicial homology of complexes associated with Stanley--Reisner rings in order to apply Reisner's criterion, which we will use to detect the virtually Cohen--Macaulay property.   We will need to recall two pieces of standard terminology and introduce one new piece of notation.  Recall that if $\sigma$ is a face of the simplicial complex $\Delta$, then we define the \emph{link of $\sigma$ in $\Delta$} to be 
\[ 
 \link_{\sigma}(\Delta) = \{\sigma' \in \Delta \mid \sigma \cup \sigma' \in \Delta, \sigma \cap \sigma' = \varnothing \}.
\] 
Recall also that $\widetilde{H}_{i}(\Delta;k)$ denotes the $i^{th}$ reduced simplicial homology of the simplicial complex $\Delta$ with coefficients in $k$.  Finally, let $\cB_r = 
\left\{\sigma \subseteq\cX \mid \dim \sigma \le r, \sigma \text{ is irrelevant}\right\}$, 
the simplicial complex of all at most $r$-dimensional irrelevant simplices.

\begin{lemma}
\label{lem:irrelevant-almost-cm}
The ring 
$S/I_{\cB_r}$ is Cohen--Macaulay on the punctured spectrum. 
Also, $\widetilde{H}_{r-2}(\cB_r;k)=k$ and $\widetilde{H}_{i}(\cB_r;k)=0$ for $i<r$ with $i\neq r-2$. 
\end{lemma}
\begin{proof}
We will show that for all $\sigma\in\cB_r\setminus\varnothing$, 
\[
\widetilde{H}_i(\link_\sigma(\cB_r);k)=0 
\quad \text{ for }
i<\dim(\link_\sigma(\cB_r)) = r-1-\dim(\sigma),
\] 
$\widetilde{H}_{r-2}(\cB_r;k)=k$, 
and $\widetilde{H}_{i}(\cB_r;k)=0$ for $i\neq r-2$ and $i<r$. 

Let $\sigma\in \cB_r$ be arbitrary. Let $\Delta = \link_{\sigma}(\cB_r)$. Now for $C\subset [r]$ with $C^{c}\cup \mathcolor(\sigma)\neq [r]$, consider the subcomplex $\Delta_C$ given by the faces of $\Delta$ that do not include the colors in $C$. Note in particular that $\Delta_C\cap \Delta_D =\Delta_{C\cup D}$.

For every face $\gamma\in \Delta$, since $\gamma\cup\sigma\in\cB_r$, it is irrelevant. Thus, there exists an $i$ such that $i\notin \mathcolor(\gamma\cup \sigma)$. In particular, $i$ satisfies both $i\notin\mathcolor(\sigma)$ and $\gamma\in \Delta_{\left\{i\right\}}$. 
Putting these together, $\left\{\Delta_{\left\{i\right\}}\right\}_{i\notin \mathcolor(\sigma)}$ provides a covering of $\Delta$, 
which induces the Mayer--Vietoris spectral sequence:  
\[
E_{p,q}^1 = 
\bigoplus_{\substack{|C|=p+1>0\\C^{c}\cup \mathcolor(\sigma)\neq [r]}} 
H_q(\Delta_C;k)
\quad \Rightarrow \quad 
H_{p+q}(\Delta;k).
\]

We claim that $\Delta_C$ is the $\left(r-1-\dim(\sigma)\right)$-skeleton of the simplex on all vertices with color in $C^{c}$, excluding those vertices in $\sigma$. 
To see this, recall that we are restricting to $C$ with $C^{c}\cup \mathcolor(\sigma)\neq [r]$. Thus for every simplicial complex $\gamma$ on $\cX$, 
with 
$\dim\gamma\le \dim(\link_{\sigma}(\cB_r))=r-1-\dim(\sigma)$ and 
$\mathcolor(\gamma)\subset C^{c}$, 
it must be that $\gamma\cup \sigma$ is irrelevant and belongs to $\cB_r$, so $\gamma\in\Delta_C$.  

Now for $\sigma\neq \varnothing$, we must show that
$\widetilde{H}_i(\Delta;k)=0$ for 
$i<\dim(\Delta) = r-1-\dim(\sigma)$ for 
$\sigma\neq\varnothing$. 
Since $\Delta_C$ is the $(r-1-\dim(\sigma))$-skeleton of a simplex, it cannot have reduced homology in degrees lower than $r-1-\dim(\sigma)$, and therefore 
$H_{q}(\Delta_C;k)=0$ for 
$0 < q < r-1-\dim(\sigma)$.
Thus for $p+q<r-1-\dim(\sigma)$ with $q\neq 0$, we have that $E_{p,q}^1=0$, as can be seen in the $E^1$ page in Figure~\ref{fig:spectral-sequence}. 
In light of this, it suffices to show that the maps on $E_{p+1,0}^1\rightarrow E_{p,0}^1$ give $0$ homology. 
But this can be observed in the total complex, since $H_0(\Delta_C;k)=k$ for $C\neq [r]$, and $\Delta_C= \varnothing$ if and only if $C=[r]$, the complex given by these maps is simply the simplicial chain complex for the nerve of the covering of $\Delta$ by $\left\{\Delta_{\left\{i\right\}}\right\}_{i\notin \mathcolor(\sigma)},$ where the nerve is the simplicial complex given by 
\[
N\left(\left\{\Delta_{\left\{i\right\}}\right\}_{i\notin\mathcolor(\sigma)}\right)=\left\{F\subset [r]\setminus \mathcolor(\sigma) \ \bigg\vert \ \bigcap_{i\in F} \Delta_{\left\{i\right\}}\neq \varnothing\right\}.
\]
In this case, this nerve is the simplex on $[r]\setminus\mathcolor(\sigma)$. Thus $E_{0,0}^2=k$ and $E_{p,q}^2=0$ for $0<p+q<r-1-\dim(\sigma)$. And therefore 
$\widetilde{H}_{i}(\Delta;k)=0$ for $i<\dim(\Delta)$. 

If instead $\sigma=\varnothing$, 
then the nerve is the boundary of the simplex on $[r]$, which is a sphere.
In light of this, the result holds by direct computation if $r\le 2$. If $r>2$, then 
$E_{0,0}^2=k$ and $E_{p,q}^2=0$ for $0<p+q<r$ except $(p,q)=(r-2,0)$ and $E_{r-2,0}^2=k$. 
Thus $\widetilde{H}_{r-2}(\Delta;k)=k$, and 
$\widetilde{H}_{i}(\Delta;k)=0$ for $i<r$ with $i\neq r-2$.

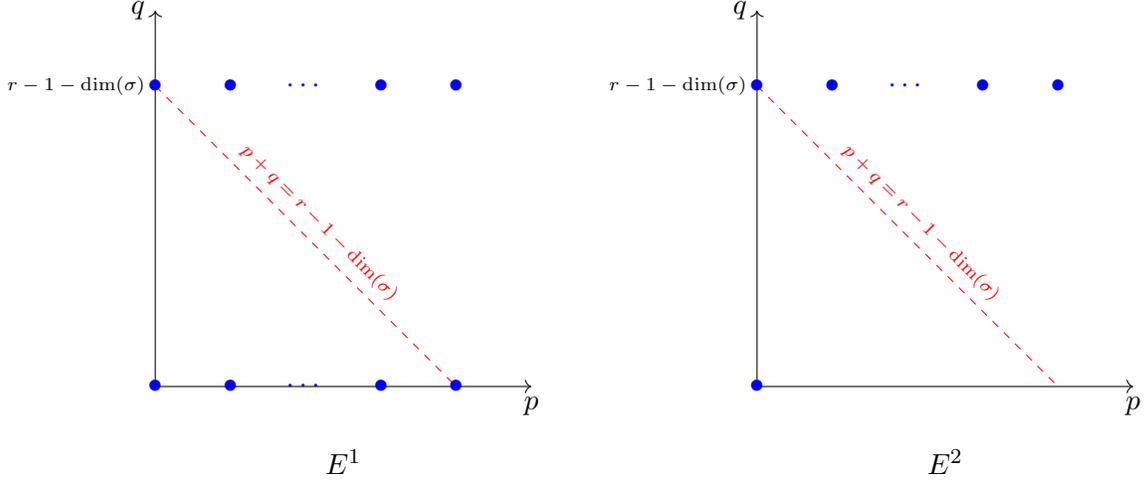
\begin{figure}
  \caption{\label{fig:spectral-sequence} An illustration of the $E^1$ page and the resulting vanishing on the $E^2$ page in the case of $\sigma\neq\varnothing$. The dots represent the potentially non-zero entries on any particular page. The vanishing of the entries on the bottom row of the $E^2$ page are due to the entries on the bottom row of the $E^1$ page forming a chain complex isomorphic to the nerve complex of the covering.}
  \begin{tikzpicture}
    \begin{scope}
    \draw[->] (0,0) -- (5,0) node [below] {$p$};
    \draw[->] (0,0) -- (0,5) node [left] {$q$};
    \draw[dashed,red] (0,4)--(4,0) node [midway,above,sloped] {\tiny$p+q=r-1-\dim(\sigma)$};
    \node [left,black] at (0,4) {\tiny $r-1-\dim(\sigma)$} ;
    \node[blue] at (0,0) {$\bullet$};
    \node[blue] at (1,0) {$\bullet$};
    \node[blue] at (2,0) {$\cdots$};
    \node[blue] at (3,0) {$\bullet$};
    \node[blue] at (4,0) {$\bullet$};

    \node[blue] at (0,4) {$\bullet$};
    \node[blue] at (1,4) {$\bullet$};
    \node[blue] at (2,4) {$\cdots$};
    \node[blue] at (3,4) {$\bullet$};
    \node[blue] at (4,4) {$\bullet$};

    \node at (2.5,-1) {$E^1$};
    
    \end{scope}

    \begin{scope}[shift={(8,0)}]
    \draw[->] (0,0) -- (5,0) node [below] {$p$};
    \draw[->] (0,0) -- (0,5) node [left] {$q$};
    \node [left,black] at (0,4) {\tiny $r-1-\dim(\sigma)$} ;
    \draw[dashed,red] (0,4)--(4,0) node [midway,above,sloped] {\tiny$p+q=r-1-\dim(\sigma)$};
    \node[blue] at (0,0) {$\bullet$};

    \node[blue] at (0,4) {$\bullet$};
    \node[blue] at (1,4) {$\bullet$};
    \node[blue] at (2,4) {$\cdots$};
    \node[blue] at (3,4) {$\bullet$};
    \node[blue] at (4,4) {$\bullet$};
    \node at (2.5,-1) {$E^2$};
    \end{scope}
  \end{tikzpicture}
\end{figure}

\end{proof}

In Lemma~\ref{lem:irrelevant-almost-cm}, the $r$ used need not be the same as $r$ in $X=\PP^{n_1}\times\cdots \times \PP^{n_r}$; however, in practice we will only need the case where the two $r$'s agree.

Our next goal is to use Lemma~\ref{lem:irrelevant-almost-cm} to prove Theorem~\ref{thm:relevant-connected-ideal-VCM}, stated below, which is a special case of Theorem~\ref{thm:monomial-vcm}, a case to which we will ultimately reduce the main theorem.  A relevant simplicial complex $\Delta$ is \emph{relevant-connected} if its geometric realization is (topologically) connected after removing the realization of any of its irrelevant faces. Further, a subcomplex of $\Delta$ is called a \emph{relevant-connected component} if it is maximal among the relevant-connected subcomplexes of $\Delta$.

\begin{theorem}
\label{thm:relevant-connected-ideal-VCM}
If $\Delta$ is an $r$-dimensional relevant-connected  simplicial complex on $\cX$, then $S/I_{\Delta}$ is virtually Cohen--Macaulay. 
\end{theorem}

 The case $r=1$ is that of a single projective space.  In this case, the complex $\Delta$ is pure of dimension $1$ and relevant-connected, and so $\Delta$ is Cohen--Macaulay by Reisner's criterion.

In order to prove Theorem~\ref{thm:relevant-connected-ideal-VCM}, we will need to introduce and study interior and exterior faces, which we do now.  Let $\Delta$ be a relevant simplicial complex and $\sigma \neq \varnothing$ a face of $\Delta$. 
Let $\mbox{Ex}(\sigma, \Delta) = \link_{\sigma}(\Delta)\cap  \link_{\sigma}(\cB_r)$, and call the faces in $\mbox{Ex}(\sigma, \Delta)$ the \emph{exterior} faces of $\link_{\sigma}(\Delta)$. Call the rest of the faces in $\link_{\sigma}(\Delta)$ the \emph{interior} faces of $\link_{\sigma}(\Delta)$.

\begin{remark}
\label{rmk:intersect-ext-int}
Note that the intersection of an exterior and an interior face is an exterior face. 
\end{remark}

\begin{example}
\label{ex:interior-exterior}
Consider the following example in $\PP^3\times \PP^3$, where in Figure~\ref{fig:PP3xPP3}
the first copy of $\PP^3$ is colored red and the second copy of $\PP^3$ is colored blue. 
Consider the link of the red square vertex, whose faces consist of the red triangle vertices, blue pentagon vertices and dashed lines. The exterior faces are the red triangle vertices and the interior faces are the blue pentagon vertices and the dashed lines. 
Notice that the blue hexagons and dashed line on the right hand edge of the diagram are irrelevant, but are still interior faces. 
\begin{figure}[h]
\centering
\caption{Each column of vertices corresponds to a copy of $\PP^3$.}
\label{fig:PP3xPP3}
  \begin{tikzpicture}
    \fill[gray] (1,-0.5)--(0,0)--(1,0.5) -- cycle;
    \fill[gray] (1,0.5)--(0,1)--(1,1.5) -- cycle;
    \fill[gray] (1,1.5)--(0,2)--(1,2.5) -- cycle;
    \fill[gray] (0,0)--(1,0.5)--(0,1) -- cycle;
    \fill[gray] (0,1)--(1,1.5)--(0,2) -- cycle;
    \fill[gray] (0,2)--(1,2.5)--(0,3) -- cycle;
    \draw (0,0)--(0,1)--(0,2)--(0,3);
    \draw (1,-0.5)--(1,0.5);
    \draw (1,1.5)--(1,2.5);
    \draw (1,-0.5)--(0,0);
    \draw (1,0.5)--(0,1)--(1,1.5);
    \draw (0,2)--(1,2.5)--(0,3);
    \draw[dashed] (0,0)--(1,0.5)--(1,1.5)--(0,2);
    
    \node[red,fill,regular polygon, regular polygon sides=3,scale=0.3] at (0,0) {$\bullet$};
    \node[red,fill, rectangle, scale=0.9] at (0,1) {$$};
    \node[red, fill,regular polygon, regular polygon sides=3,scale=0.3] at (0,2) {$\bullet$};
    \node[red] at (0,3) {$\bullet$};
    \node[blue] at (1,-0.5) {$\bullet$};
    \node[blue,fill,regular polygon, regular polygon sides=5,scale=0.7] at (1,0.5) {$$};
    \node[blue,fill,regular polygon, regular polygon sides=5,scale=0.7] at (1,1.5) {$$};
    \node[blue] at (1,2.5) {$\bullet$};
  \end{tikzpicture}
\end{figure}
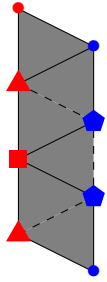

The idea of interior and exterior faces can become considerably more complex. Consider the following illustration of the link of a cell in an example $\Delta$ on $\PP^n\times \PP^m\times \PP^{\ell}$. 
In Figure~\ref{fig:octagon1}, the vertices corresponding to each of the parts of the product are colored red, blue, and green. Only the link is illustrated, and it is the link of a vertex that would be colored blue.  Therefore the bold faces are the exterior faces, and the others are interior faces.
\begin{figure}[h!]
\centering
\caption{The link in some $\Delta$ of a certain blue vertex on $\PP^n\times \PP^m\times \PP^{\ell}$.}
\label{fig:octagon1}
    \begin{tikzpicture}[scale=1.5]
    \coordinate (a) at (0,0);
    \node[blue] at (a) {$\bullet$};
    \coordinate (b) at (1,-1);
    \node[red] at (b) {$\bullet$};
    \coordinate (c) at (2,-1);
    \node[red] at (c) {$\bullet$};
    \coordinate (d) at (3,0);
    \node[red] at (d) {$\bullet$};
    \coordinate (e) at (3,1);
    \node[blue] at (e) {$\bullet$};
    \coordinate (f) at (2,2);
    \node[green] at (f) {$\bullet$};
    \coordinate (g) at (1,2);
    \node[green] at (g) {$\bullet$};
    \coordinate (h) at (0,1);
    \node[green] at (h) {$\bullet$};

    \coordinate (x1) at (1,1);
    \node[red] at (x1) {$\bullet$};
    \coordinate (x2) at (1,0);
    \node[green] at (x2) {$\bullet$};
    \coordinate (x3) at (1.5,0);
    \node[red] at (x3) {$\bullet$};
    \coordinate (x4) at (2,1);
    \node[green] at (x4) {$\bullet$};

    \coordinate (x5) at (2,1.5);
    \node[red] at (x5) {$\bullet$};

    \coordinate (x6) at (1.5,1.25);
    \node[green] at (x6) {$\bullet$};

    \draw[ultra thick] (a) -- (b) -- (c) -- (d) -- (e) -- (f) -- (g) -- (h) -- cycle;

    \draw (h)--(b);
    \draw (x1)--(b);
    \draw (d)--(f);
    
    \draw (h) -- (x2) -- (x1) -- (x4) -- (x3) -- (x2) -- (c);
    \draw (c) -- (x3);
    \draw (c) -- (x4) -- (d);
    \draw (x4) -- (x5) -- (f) -- (x1);
    \draw (x1) -- (x6) -- (x5) -- (d);
    \draw (x1) -- (h);
    \draw (x1) -- (g);
    \draw (x1)--(x3);
    \draw (f)--(x6)--(x4);

    \draw[ultra thick] (c)--(x3)--(x1);
    \draw[ultra thick] (h) -- (x2);
    \draw[ultra thick] (f) -- (x6) -- (x4);
    \draw[ultra thick] (d) -- (x5);
    
    \end{tikzpicture}
\end{figure}
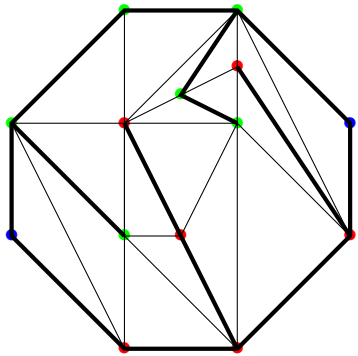
\end{example}

\begin{lemma}
\label{lem:twoface}
Let $\Delta$ be a pure, relevant $r$-dimensional simplicial complex. If $\sigma \neq \varnothing$ is a simplex of $\Delta$, then  every facet of $\link_{\sigma}(\Delta)$ has at most two codimension $1$ faces that are interior faces of $\link_{\sigma}(\Delta)$. Moreover,
  \vspace{-1.5mm}
  \begin{enumerate}
  \item A facet of $\link_{\sigma}(\Delta)$ has no codimension $1$ faces that are interior if and only if $\sigma$ uses some color at least twice. 
  \item A facet $\tau$ of $\link_{\sigma}(\Delta)$ has exactly one codimension $1$ face that is interior if and only if $\tau$ shares a color with $\sigma$.
  \item A facet $\tau$ of $\link_{\sigma}(\Delta)$ has exactly two codimension $1$ faces that are interior if and only if $\tau$ uses some color at least twice. 
  \end{enumerate}
\end{lemma}
\begin{proof}
Let $\tau$ be a facet of $\link_{\sigma}(\Delta)$. 
By assumption, $\tau\cup\sigma$ is relevant and $\dim(\tau\cup\sigma) = r$. 
Since $\tau\cup\sigma$ is relevant, it has color $[r]$; since $\tau\cup\sigma$ has dimension $r$, it has exactly $r+1$ vertices. 
Putting these two facts together, it follows that in $\tau\cup\sigma$ exactly one color is used twice. The rest of the argument proceeds by carefully considering the locations of that twice-used color.

Let the vertices of the twice-used color in $\tau\cup\sigma$ be labeled $v_1$ and $v_2$. 
Now, a codimension 1 face $\gamma$ of $\tau$ is an interior face if and only if 
$\tau\setminus\gamma\subset \left\{v_1,v_2\right\}$,
since both of these conditions are the same as requiring that $\gamma\cup\sigma$ be colored by $[r]$. 
But this immediately implies that $\tau$ contains at most two codimension $1$ faces that are interior. 

Moving now to the exact number of codimension $1$ faces of $\tau$ that are interior faces, 
we consider which of $\tau$ or $\sigma$ contains each of the vertices $v_1$ and $v_2$.
\vspace{-1.5mm}
\begin{enumerate}
\item There are no codimension $1$ faces of $\tau$ that are interior faces if and only if $v_1,v_2\in\sigma$, which is equivalent to $\sigma$ using some color twice. 

\item There is one codimension $1$ face of $\tau$ that is an interior face if and only if $v_i\in\sigma$ for precisely one $i\in\{1,2\}$. In this case, for $j\neq i$, we have $v_j\in \tau$. Therefore, $\mathcolor(\sigma)\cap\mathcolor(\tau)\neq \varnothing$.

\item There are two codimension $1$ faces of $\tau$ that are interior faces if and only if both $v_1,v_2\in\tau$. This is equivalent to $\tau$ using some color at least twice.
\qedhere
\end{enumerate}
\end{proof}

We are now prepared to prove Theorem~\ref{thm:relevant-connected-ideal-VCM}.  The proof makes heavy use of Reisner's criterion, which we record below for convenience.  Reisner showed in his thesis that $S/I_\Delta$ is Cohen--Macaulay if and only if $\Delta$ is Cohen--Macaulay as a simplicial complex.  It is for this reason, combined with the statement of Reisner's criterion, that the proof of Theorem~\ref{thm:relevant-connected-ideal-VCM} centers on the computation of reduced simplicial homology.

\begin{theorem}[Reisner's Criterion]
A simplicial complex $\Delta$ is Cohen--Macaulay if and only if $\widetilde{H}_i(\link_\sigma(\Delta);k) = 0$ for all $i<\dim \link_\sigma(\Delta)$.
\end{theorem}

\begin{proof}[Proof of Theorem~\ref{thm:relevant-connected-ideal-VCM}]
First note that we may assume that all facets of $\Delta$ are relevant. We claim that all $r$-dimensional relevant-connected simplicial complexes having no irrelevant facets are pure. Note that an $r$-dimensional simplicial complex corresponds to a $1$-dimensional subvariety of $\PP^{\mathbf{n}}$ and relevant-connectivity of $\Delta$ implies connectivity of the corresponding variety. Since all connected 1-dimensional subvarieties of $\PP^{\mathbf{n}}$ are equidimensional, there cannot be any lower dimensional relevant facets of $\Delta$. Then since $\Delta$ contains only relevant facets, $\Delta$ is pure.
   
We will produce a simplicial complex $\Delta'$ that is Cohen--Macaulay and differs from $\Delta$ in only irrelevant faces.
Note the theorem is trivially true when $r=1$, 
since this is the case of a single projective space, and, in this case, $\Delta$ must be a 1-dimensional pure and relevant-connected simplicial complex. It now follows easily from Reisner's criterion that $\Delta$ is Cohen--Macaulay. 

When $r>1$, there are two cases. 
First, consider the case that $\Delta$ is of the form $\Delta=\join(\tau,\Omega)$ for a face $\tau\neq \varnothing$ and a simplicial subcomplex $\Omega$ with $\mathcolor(\tau)\cap\mathcolor(\Omega)=\varnothing$.
Since $\tau$ is a simplex, there is a bijection between the top-dimensional cells of $\Delta$ and those of $\Omega$.
Thus, since $\Delta$ is relevant-connected, $\Omega$ is, too. Further, since $\dim \Delta=r$,  $\dim \Omega = r-\dim \tau-1$. Moreover, since any face of $\Delta$ uses at most one color twice, we know that either $\dim\tau = \left|\mathcolor(\tau)\right|$ or $\dim\tau = \left|\mathcolor(\tau)\right|-1$. In the first case, we find that $\dim\Omega = \left|\mathcolor(\Omega)\right|-1$. Now restricting to the colors in $\mathcolor(\Omega)$ and applying \cite[Theorem 1.3]{reu2019} to $\Omega$, we can construct a Cohen--Macaulay simplicial complex $\Omega'$ that differs from $\Omega$ only on irrelevant faces.

On the other hand, if $\dim\sigma = \left|\mathcolor(\sigma)\right|-1$, then $\dim\Omega = \left|\mathcolor(\Omega)\right|-1$ and $\dim\Omega<r$ so by replacing $\Delta$ with $\Omega$ and by using induction on $r$, we will construct a Cohen--Macaulay simplicial complex $\Omega'$ differing from $\Omega$ only on irrelevant faces. 

Now we take the simplicial complex $\Omega'$ and let $\Delta'=\join(\sigma,\Omega')$. Then $\Delta'$ differs from $\Delta$ only in irrelevant faces.
Since $\sigma$ is a simplex, $\join(\sigma,\Omega')$ can be constructed by iteratively taking the cone over $\Omega'$ by the vertices in $\sigma$. Since the cone over a Cohen--Macaulay simplicial complex is Cohen--Macaulay, $\Delta'$ is Cohen--Macaulay, and so $\Delta$ is virtually Cohen--Macaulay.

For the second and final case, suppose that $\Delta$ is not of the form $\Delta=\join(\tau,\Omega)$, where $\tau\neq\varnothing$ is a face of $\Delta$ and $\mathcolor(\tau)\cap\mathcolor(\Omega)=\varnothing$. 
Then, define $\Delta'=\Delta\cup \cB_r$. 
We claim that $\Delta'$ is Cohen--Macaulay, and we will show this using Reisner's criterion, i.e.,  we will show that 
\begin{equation}
\label{eqn:Reisner}
\widetilde{H}_{i}(\link_{\sigma}(\Delta');k)=0 
\quad\text{for each face $\sigma$ of $\Delta'$ and all $i<d = \dim\link_{\sigma}(\Delta')$}.
\end{equation}
Since $\Delta'=\Delta\cup \cB_{r}$,
it follows that 
$\link_{\sigma}(\Delta') 
= \link_{\sigma}(\Delta) \cup \link_{\sigma}(\cB_r).
$
Then the long exact sequence of a pair yields the exact sequence 
\begin{equation}\label{eqn:les-pair}
\widetilde{H}_{i}( \link_{\sigma}(\cB_r);k)\rightarrow 
\widetilde{H}_{i}(\link_{\sigma}(\Delta');k)\rightarrow 
H_{i}(\link_{\sigma}(\Delta'), \link_{\sigma}(\cB_r);k)\rightarrow
\widetilde{H}_{i-1}( \link_{\sigma}(\cB_r);k).
\end{equation}
Then for any $i$, so long as $\widetilde{H}_{i}(\link_{\sigma}(\cB_{r});k)=\widetilde{H}_{i-1}(\link_{\sigma}(\cB_{r});k)=0$, it suffices to show
\begin{equation}
\label{eqn:Reisner-relative}
H_{i}(\link_{\sigma}(\Delta'), \link_{\sigma}(\cB_r);k)=0.
\end{equation}
We will first treat the case of $\sigma \neq \varnothing$ and then separately handle the case of $\sigma = \varnothing$. 

For $\sigma\neq\varnothing$, 
since $H_{i}( \link_{\sigma}(\cB_r);k)=0$ for $i<d$, by Lemma~\ref{lem:irrelevant-almost-cm} and thanks to~\eqref{eqn:Reisner-relative}, 
it suffices to show that 
$H_i(\link_{\sigma}(\Delta'), \link_{\sigma}(\cB_r);k)=0$ for all $i<d$. 
Notice that 
\[
H_i(\link_{\sigma}(\Delta'), \link_{\sigma}(\cB_r);k) 
= H_{i}(\link_\sigma(\Delta),\mbox{Ex}(\sigma, \Delta);k),
\] 
where $\mbox{Ex}(\sigma, \Delta)= \link_{\sigma}(\Delta)\cap\link_{\sigma}(\cB_r)$. 
To complete the proof for the case $\sigma\neq \varnothing$, we will show that $H_{i}(\link_\sigma(\Delta),\mbox{Ex}(\sigma, \Delta);k)=0$, starting with 
$i<d-2$. We will then treat separately the cases $i=d-2$ and $i=d-1$. 

When $i<d-2$, 
let $\tau$ be an $i$-face of $\link_{\sigma}(\Delta)$ of codimension at least $3$. 
Then we claim that $\tau$ is an exterior face. 
To see this, let $\tilde{\tau}$ be a facet of $\link_{\sigma}(\Delta)$ that contains $\tau$. 
Because $\tau$ is of codimension at least $3$ in $\tilde{\tau}$, it is contained in at least $3$ codimension $1$ faces of $\tilde{\tau}$, and therefore, by Lemma~\ref{lem:twoface}, it is contained in at least one exterior facet of $\tilde{\tau}$. 
Hence $\tau\in \mbox{Ex}(\sigma, \Delta)$, so  
$C_i(\link_{\sigma}(\Delta),\mbox{Ex}(\sigma, \Delta);k)=0$ and 
$H_i(\link_{\sigma}(\Delta),\mbox{Ex}(\sigma, \Delta);k)=0$ for $i<d-2$, as desired.

When $i=d-2$, we must show that  
$H_{d-2}(\link_{\sigma}(\Delta),\mbox{Ex}(\sigma, \Delta);k)=0$. 
To do so, we will again show that every $(d-2)$-face $\tau$ in $\link_{\sigma}(\Delta)$ is a boundary relative to $\mbox{Ex}(\sigma, \Delta)$. 
Without loss of generality, assume that $\tau$ is not in $\mbox{Ex}(\sigma, \Delta)$. 
Let $\tilde{\tau}$ be a facet of $\link_{\sigma}(\Delta)$ containing $\tau$. 
Since $\tau$ is of codimension $2$ in $\link_{\sigma}(\Delta)$, $\tau$ is contained in exactly two codimension $1$ faces of $\tilde{\tau}$. 
Further, since $\tau$ is not in $\mbox{Ex}(\sigma, \Delta)$, it must be that both of these codimension $1$ faces of $\tilde{\tau}$ are interior faces; call one of them $\xi$. 
By Remark~\ref{rmk:intersect-ext-int}, the other codimension $1$ faces of $\xi$ besides $\tau$ must be in $\mbox{Ex}(\sigma, \Delta)$. 
Therefore, up to sign, the relative boundary with respect to $\mbox{Ex}(\sigma, \Delta)$ of $\xi$ is $\tau$. 
Since $\tau$ was arbitrary, 
$H_{d-2}(\link_{\sigma}(\Delta),\mbox{Ex}(\sigma, \Delta);k)=0$, as desired. 

Finally, when $i=d-1$, we must show that $H_{d-1}(\link_{\sigma}(\Delta),\mbox{Ex}(\sigma, \Delta);k)=0$. 
To do so, we will use Lemma~\ref{lem:twoface} to construct a graph (with loops). In the graph $G$, there is a distinguished vertex $*$, while the other vertices correspond to the codimension $1$ faces of $\link_{\sigma}(\Delta)$ that are interior. 
Edges are placed to connect vertices corresponding to interior faces that are both contained in a  common facet of $\link_{\sigma}(\Delta)$. 
When a facet of $\link_{\sigma}(\Delta)$ has one codimension $1$ face that is interior, then an edge is placed between the vertex for that facet and $*$. Finally, if a facet of $\link_{\sigma}(\Delta)$ has no interior faces, a loop is placed at $*$. 

Recall that we are currently in the case that $\Delta$ is not of the form $\join(\tau,\Omega)$, where $\tau$ is a face of $\Delta$ and $\mathcolor(\tau)\cap\mathcolor(\Omega)=\varnothing$. 
We claim that, in this case, $\link_{\sigma}(\Delta)$ contains at least one facet that has at most one codimension $1$ face that is interior. 
To see this, by way of contradiction, suppose that in $\link_{\sigma}(\Delta)$, all facets contain exactly two codimension 1 faces that are interior. 
Let $\tau\in \link_{\sigma}(\Delta)$ be such a facet, in which case $\tau\cup\sigma$ is a facet of $\Delta$. Since $\tau$ has exactly two codimension $1$ faces that are interior, by Lemma~\ref{lem:twoface}, $\tau$ uses some color at least twice. 
Then, since $\tau\cup\sigma$ contains $r+1$ vertices and there are only $r$ possible colors, it must be that the colors used in $\sigma$ are present only in $\sigma$ and not in $\tau$.
But since a relevant simplex must contain all colors, the relevant facets 
of $\tau\cup\sigma$ must all contain $\sigma$. 
Further, since $\Delta$ is relevant-connected, repeating this for the successive neighbors of $\tau\cup\sigma$ in $\Delta$, we find that all facets of $\Delta$ contain $\sigma$, and thus $\Delta=\join(\sigma,\Omega)$ for some $\Omega$, a contradiction. 
Therefore, it must be that $\link_{\sigma}(\Delta)$ contains at least one facet for which at most one of its codimension $1$ faces is interior. 

By the previous paragraph, the graph $G$ is connected, 
and there is a commutative diagram
\begin{equation}
\label{eqn:graph-commutative-diagram}
  \begin{tikzcd}
C_{d}(\link_{\sigma}(\Delta),\mbox{Ex}(\sigma, \Delta);k) \arrow[r]\arrow[d,"\cong"]
&
C_{d-1}(\link_{\sigma}(\Delta),\mbox{Ex}(\sigma, \Delta);k)\arrow[d,"\cong"]
\\
C_1(G,*;k) \arrow[r,twoheadrightarrow]
&
C_0(G,*;k).
\end{tikzcd}
\end{equation}
The surjectivity of the bottom map in \eqref{eqn:graph-commutative-diagram} is a consequence of the fact that $G$ is connected, so  $H_0(G,*,k)=0$. 
Since the vertical maps in the diagram are isomorphisms, the top map in \eqref{eqn:graph-commutative-diagram} is also surjective. Therefore, $H_{d-1}(\link_{\sigma}(\Delta),\mbox{Ex}(\sigma, \Delta);k)=0$, which concludes the proof of \eqref{eqn:Reisner} for any face $\sigma\neq \varnothing$ in $\Delta'$. 

It now remains to show that condition \eqref{eqn:Reisner} holds for $\sigma=\varnothing$. 
Before beginning this portion of the proof, note that the argument in the cases that $\sigma\neq\varnothing$ and $i\le d-2$ case above apply here as well to show that 
\begin{equation}
\label{eqn:rel-vng}
H_{i}(\Delta',\cB_r;k)=0 \quad\text{for}\quad i\le r-2. 
\end{equation}
Now consider the case that $\sigma=\varnothing$ and $i<r-2$, where $r=\dim(\Delta')$. 
By Lemma~\ref{lem:irrelevant-almost-cm}, $H_{i}(\cB_r;k)=0$ for $i<r-2$. 
Putting this together with \eqref{eqn:rel-vng}, it now follows from \eqref{eqn:les-pair} that $\widetilde{H}_i(\Delta';k)=0$ for all $i<r-2$. 

It remains to show that condition \eqref{eqn:Reisner} holds for $\sigma=\varnothing$ in the cases $i=r-2$ and $i=r-1$. 
The long exact sequence of a pair together with \eqref{eqn:rel-vng} yield the exact sequence: 
\[
H_{r-1}(\cB_r;k)\rightarrow 
H_{r-1}(\Delta';k)\rightarrow 
H_{r-1}(\Delta',\cB_r;k)\rightarrow 
H_{r-2}(\cB_r;k)\rightarrow 
H_{r-2}(\Delta';k)\rightarrow 
0. 
\]
Applying Lemma~\ref{lem:irrelevant-almost-cm}, 
this simplifies to 
\begin{equation}
  \label{eqn:empty-simplex-exact-sequence}
0\rightarrow 
H_{r-1}(\Delta';k)\rightarrow 
H_{r-1}(\Delta',\cB_r;k)
\rightarrow k \rightarrow 
H_{r-2}(\Delta';k)\rightarrow 0. 
\end{equation}
Thus, it suffices to show that 
$H_{r-1}(\Delta',\cB_r;k)=k$ and that 
the map $H_{r-1}(\Delta';k)\rightarrow 
H_{r-1}(\Delta',\cB_r;k)$ 
is the zero map, since this would imply that the map $H_{r-1}(\Delta',\cB_r;k)\rightarrow k$ is an isomorphism, so that $H_{r-1}(\Delta';k)=H_{r-2}(\Delta';k)=0$, as desired.

To see that $H_{r-1}(\Delta',\cB_r;k)=k$, 
note first that the codimension $1$ faces of any $(r-1)$-simplex are irrelevant for dimension reasons. Thus, the boundary of any $(r-1)$-face in $\Delta'$ belongs to $\cB_r$, so the $(r-1)$-faces of $\Delta$ provide a generating set for $H_{r-1}(\Delta',\cB_r;k)$. 
By Lemma~\ref{lem:twoface}, every $r$-dimensional relevant face of $\Delta$ has precisely two relevant $(r-1)$-faces in its boundary. 
Further, since $\Delta$ is relevant-connected, 
every relevant $(r-1)$-face of $\Delta'$ is nonzero and homologically equivalent up to sign. 
Therefore, $H_{r-1}(\Delta',\cB_r;k)=k$ and any relevant $(r-1)$-face of $\Delta'$ gives a generator of this homology group. 

Finally, to see that 
$H_{r-1}(\Delta';k)\rightarrow 
H_{r-1}(\Delta',\link_{\sigma}(\cB_r))$ 
is the zero map, 
let $\Gamma$ be a simplex on the color set $[r]$. 
There is a projection $\Delta'\rightarrow \Gamma$, given by mapping all vertices of a given color $i$ onto the vertex $i$. 
Since any relevant $(r-1)$-face of $\Delta'$ gives a generator of $H_{r-1}(\Delta',\cB_r;k)=k$ and the induced map sends this generator to a nonzero element of $H_{r-1}(\Gamma,\partial \Gamma) = k$, this induced map $H_{r-1}(\Delta',\cB_r)\rightarrow H_{r-1}(\Gamma,\partial \Gamma)$ is an isomorphism.

Now consider the following diagram induced by the map $\Delta'\rightarrow \Gamma$
\[
\begin{tikzcd}
  H_{r-1}(\Delta') \arrow[r] \arrow[d] & H_{r-1}(\Gamma) = 0 \arrow[d] \\
  H_{r-1}(\Delta',\cB_r) \arrow[r,"\cong"] & H_{r-1}(\Gamma,\partial \Gamma).
\end{tikzcd}
\]
This diagram commutes, and, thus, the map $H_{r-1}(\Delta')\rightarrow H_{r-1}(\Delta',\cB_r)$ is the zero map. Applying this fact to the exact sequence \eqref{eqn:empty-simplex-exact-sequence}, we get $H_{r-1}(\Delta';k)=H_{r-2}(\Delta';k)=0$.
\end{proof}

\begin{example}
Continuing with Example~\ref{ex:interior-exterior}, one of the critical steps in the proof of Theorem~\ref{thm:relevant-connected-ideal-VCM} is the reduction of some of the more troublesome homology groups (in the case that $\sigma\neq\varnothing$ and $i=d-1$) to the homology of a graph by the construction of the graph given by the interior faces of the link. It is Lemma~\ref{lem:twoface} that allows such a graph to be constructed. In Figure~\ref{fig:octagon2} that graph is shown with the vertices given by $\times$ symbols, the edges given by dashed lines, and the half edges are illustrated with an edge terminated with a $\circ$ symbol.
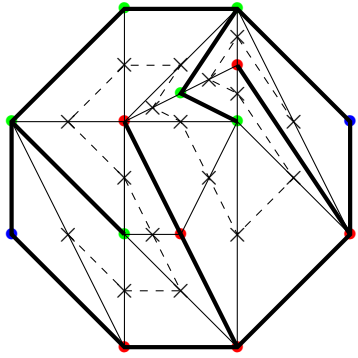
\begin{figure}[h!]
\centering
\caption{The graph associated to the complex of interior faces of the link from Figure~\ref{fig:octagon1}.}
\label{fig:octagon2}
    \begin{tikzpicture}[scale=1.5]
    \coordinate (a) at (0,0);
    \node[blue] at (a) {$\bullet$};
    \coordinate (b) at (1,-1);
    \node[red] at (b) {$\bullet$};
    \coordinate (c) at (2,-1);
    \node[red] at (c) {$\bullet$};
    \coordinate (d) at (3,0);
    \node[red] at (d) {$\bullet$};
    \coordinate (e) at (3,1);
    \node[blue] at (e) {$\bullet$};
    \coordinate (f) at (2,2);
    \node[green] at (f) {$\bullet$};
    \coordinate (g) at (1,2);
    \node[green] at (g) {$\bullet$};
    \coordinate (h) at (0,1);
    \node[green] at (h) {$\bullet$};

    \coordinate (x1) at (1,1);
    \node[red] at (x1) {$\bullet$};
    \coordinate (x2) at (1,0);
    \node[green] at (x2) {$\bullet$};
    \coordinate (x3) at (1.5,0);
    \node[red] at (x3) {$\bullet$};
    \coordinate (x4) at (2,1);
    \node[green] at (x4) {$\bullet$};

    \coordinate (x5) at (2,1.5);
    \node[red] at (x5) {$\bullet$};

    \coordinate (x6) at (1.5,1.25);
    \node[green] at (x6) {$\bullet$};

    \draw[ultra thick] (a) -- (b) -- (c) -- (d) -- (e) -- (f) -- (g) -- (h) -- cycle;

    \draw (h)--(b);
    \draw (x1)--(b);
    \draw (d)--(f);
    
    \draw (h) -- (x2) -- (x1) -- (x4) -- (x3) -- (x2) -- (c);
    \draw (c) -- (x3);
    \draw (c) -- (x4) -- (d);
    \draw (x4) -- (x5) -- (f) -- (x1);
    \draw (x1) -- (x6) -- (x5) -- (d);
    \draw (x1) -- (h);
    \draw (x1) -- (g);
    \draw (x1)--(x3);
    \draw (f)--(x6)--(x4);

    \draw[ultra thick] (c)--(x3)--(x1);
    \draw[ultra thick] (h) -- (x2);
    \draw[ultra thick] (f) -- (x6) -- (x4);
    \draw[ultra thick] (d) -- (x5);

    \coordinate (v1) at ($(h)!0.5!(b)$);
    \coordinate (v2) at ($(x2)!0.5!(b)$);
    \coordinate (v2p5) at ($(x2)!0.5!(c)$);
    \coordinate (v3) at ($(x2)!0.5!(x3)$);
    \coordinate (v4) at ($(x2)!0.5!(x1)$);
    \coordinate (v5) at ($(h)!0.5!(x1)$);
    \coordinate (v6) at ($(g)!0.5!(x1)$);
    \coordinate (v7) at ($(f)!0.5!(x1)$);
    \coordinate (v8) at ($(x6)!0.5!(x1)$);
    \coordinate (v9) at ($(x4)!0.5!(x1)$);
    \coordinate (v10) at ($(x4)!0.5!(x3)$);
    \coordinate (v11) at ($(x4)!0.5!(c)$);
    \coordinate (v12) at ($(x4)!0.5!(d)$);
    \coordinate (v13) at ($(x4)!0.5!(x5)$);
    \coordinate (v14) at ($(x6)!0.5!(x5)$);
    \coordinate (v15) at ($(f)!0.5!(x5)$);
    \coordinate (v16) at ($(f)!0.5!(d)$);

    \node at (v1) {$\times$};
    \node at (v2) {$\times$};
    \node at (v2p5) {$\times$};
    \node at (v3) {$\times$};
    \node at (v4) {$\times$};
    \node at (v5) {$\times$};
    \node at (v6) {$\times$};
    \node at (v7) {$\times$};
    \node at (v8) {$\times$};
    \node at (v9) {$\times$};
    \node at (v10) {$\times$};
    \node at (v11) {$\times$};
    \node at (v12) {$\times$};
    \node at (v13) {$\times$};
    \node at (v14) {$\times$};
    \node at (v15) {$\times$};
    \node at (v16) {$\times$};

    \draw[dashed] (v1)--(v2)--(v2p5)--(v3)--(v4)--(v5)--(v6)--(v7)--(v8)--(v9)--(v10)--(v11)--(v12)--(v13)--(v14)--(v15)--(v16);
    \end{tikzpicture}
\end{figure}
\end{example}

In light of Theorem~\ref{thm:relevant-connected-ideal-VCM}, to complete the proof of Theorem~\ref{thm:monomial-vcm}
it remains to show that it is enough to show that $S/I_\Delta$ is virtually Cohen--Macaulay on each of the components of its support. 

\begin{prop}
\label{prop:disjoint-support}
Let $S$ be the Cox ring of a smooth projective toric variety $X$, and let $M$ be a finitely generated $\Pic(X)$-graded $S$-module. 
If $M$ is module with equidimensional support $\mathcal{X}=\bigsqcup X_i$ with disjoint components $X_i$, then $M$ is virtually Cohen--Macaulay if each $M|_{X_i}$ is virtually Cohen--Macaulay.
\end{prop}
\begin{proof}
Let $N = \bigoplus M|_{X_i}$. Then we claim that $\widetilde{M} \cong \widetilde{N}$. Since $N$ is a direct sum, we can decompose $\widetilde{N}$ as
\[
\widetilde{N}=\bigoplus \widetilde{M|_{X_i}}.
\] 
And since $ \mathcal{X}=\bigsqcup X_i$, we have 
\[
\widetilde{M}=\bigoplus\widetilde{M|_{X_i}}.
\]
Thus a virtual resolution of $N$ is a virtual resolution of $M$. Since $M|_{X_i}$ is virtually Cohen--Macaulay, $\vdim M|_{X_i} =\codim M|_{X_i}$. Since $ \mathcal{X}$ is equidimensional, we have that $\vdim M|_{X_i}=\codim M$. Finally, a direct sum of virtual resolutions is a virtual resolution of the direct sum, so $\vdim M=\vdim N=\codim M$.
\end{proof}

\begin{proof}[Proof of Theorem~\ref{thm:monomial-vcm}]
This result is now an immediate consequence of Theorem~\ref{thm:relevant-connected-ideal-VCM} and Proposition~\ref{prop:disjoint-support}, where $M$ in the proposition is $S/I_\Delta$ and the $X_i$ correspond to the relevant-connected components of $\Delta$. 
\end{proof}

\begin{example}
\label{ex:disjoint-lines}
Let $k[x_0,\ldots,x_3]$ be the Cox ring of $X=\PP^3$ and consider the ideal 
\[
J = \<x_0x_2,x_0x_3,x_1x_2,x_1x_3\>, 
\]
for which $S/J$ has free resolution 
\[
S^{1}
\xleftarrow{ 
	\begin{bmatrix}
    x_0x_2 & x_1x_2 & x_0x_3 & x_1x_3
    \end{bmatrix}
    }
S^{4}
\xleftarrow{ 
	\begin{bmatrix}
    -x_1 & 0 & -x_3 & 0 \\
    \phantom{-}x_0 & 0 & 0 & -x_3 \\
    0 & -x_1 & \phantom{-}x_2 & 0 \\
    0 & \phantom{-}x_0 & 0 & \phantom{-}x_2
    \end{bmatrix}
    }
S^{4}
\xleftarrow{ 
	\begin{bmatrix}
    \phantom{-}x_3 \\
    -x_2 \\
    -x_1 \\
    \phantom{-}x_0
    \end{bmatrix}
    }
S^{1}
\xleftarrow{}
0.
\]
Note that $J$ corresponds to a 1-dimensional simplicial complex with a single color, so Theorem~\ref{thm:monomial-vcm} implies that $S/J$ is virtually Cohen--Macaulay, with a short virtual resolution of the form 
\[
S^{2}
\xleftarrow{ 
	\begin{bmatrix}
    x_0 & x_1 & 0 & 0 \\
    0 & 0 & x_2 & x_3
    \end{bmatrix}
    }
S^{4}
\xleftarrow{ 
	\begin{bmatrix}
      -x_1 & 0 \\
      \phantom{-}x_0 & 0 \\
      0 & -x_3 \\
      0 & \phantom{-}x_2
      \end{bmatrix}
      }
S^{2}
\xleftarrow{}
0.
\]

See Example~\ref{ex:disjoint-lines-P4} for a discussion of the subscheme of $\PP^d$ cut out by $J$ when $d>3$.  
\end{example}

\begin{example}
\label{ex:monomial-nonCM-components}
Let $X$ =  $\PP^2\times \PP^2$, and consider the  simplicial complex $\Delta$ that is homeomorphic to a a cylinder, as shown in Figure~\ref{fig:cyl1}.
\begin{figure}[h!]
\centering
\caption{A cylindrical $\Delta$ on $\PP^2\times\PP^2$.}
\label{fig:cyl1}
  \begin{tikzpicture}
    \node (y0) at (1.9,1.5) {$y_0$};
    \node (y1) at (2.5,0) {$y_1$};
    \node (y2) at (2.5,3) {$y_2$};
    \node (x0) at (-0.6,1.5) {$x_0$};
    \node (x1) at (0,0) {$x_1$};
    \node (x2) at (0,3) {$x_2$};

    \draw (x0)--(y0);
    \draw (x1)--(y1);
    \draw (x2)--(y2);
    \draw (x2) -- (x0)--(x1);
    \draw[dotted] (x1)--(x2);
    \draw (y0)--(y1)--(y2)--(y0);
    \draw (x0)--(y1);
    \draw[dotted] (x1)--(y2);
    \draw (x2)--(y0);

    \fill[gray,opacity=0.2] (x0.center)--(y1.center)--(x1.center)--cycle;
    \fill[gray,opacity=0.2] (y1.center)--(y0.center)--(x0.center)--cycle;
    \fill[gray,opacity=0.2] (x0.center)--(y0.center)--(x2.center)--cycle;
    \fill[gray,opacity=0.2] (y2.center)--(y0.center)--(x2.center)--cycle;
    \fill[gray,opacity=0.2] (x1.center)--(x2.center)--(y2.center)--cycle;
    \fill[gray,opacity=0.2] (x1.center)--(y1.center)--(y2.center)--cycle;
  \end{tikzpicture}
\end{figure}
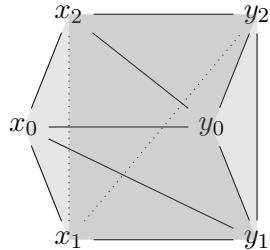

The Stanley--Reisner ideal corresponding to $\Delta$ is $I_{\Delta}=\left<x_0y_2,x_1y_0,x_2y_1,x_0x_1x_2,y_0y_1y_2\right\>$. Since $\widetilde{H}_1(\Delta;k)\neq 0$ and $\dim\Delta=2$, Reisner's criterion implies that $S/I_{\Delta}$ is not Cohen--Macaulay. On the other hand, if we consider the simplicial complex given by $\cB_{2}\cup \Delta$, which is illustrated in Figure~\ref{fig:cyl2} and corresponds to the ideal $J=\left<x_0y_2,x_1y_0,x_2y_1\right\>$, then one can check that Reisner's criterion is satisfied in this case. Since $\widetilde{I_{\Delta}}=\widetilde{J}$, we conclude that $S/I_{\Delta}$ is virtually Cohen--Macaulay.

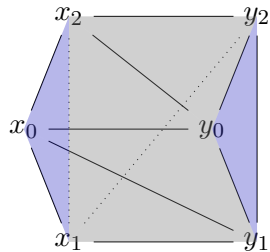
\begin{figure}[h]
\centering  
\caption{Adding irrelevant faces to $\Delta$ in Figure~\ref{fig:cyl1} yields a Cohen--Macaulay  complex. }
\label{fig:cyl2}
  \begin{tikzpicture}
    \node (y0) at (1.9,1.5) {$y_0$};
    \node (y1) at (2.5,0) {$y_1$};
    \node (y2) at (2.5,3) {$y_2$};
    \node (x0) at (-0.6,1.5) {$x_0$};
    \node (x1) at (0,0) {$x_1$};
    \node (x2) at (0,3) {$x_2$};

    \draw (x0)--(y0);
    \draw (x1)--(y1);
    \draw (x2)--(y2);
    \draw (x2) -- (x0)--(x1);
    \draw[dotted] (x1)--(x2);
    \draw (y0)--(y1)--(y2)--(y0);
    \draw (x0)--(y1);
    \draw[dotted] (x1)--(y2);
    \draw (x2)--(y0);

    \fill[gray,opacity=0.2] (x0.center)--(y1.center)--(x1.center)--cycle;
    \fill[gray,opacity=0.2] (y1.center)--(y0.center)--(x0.center)--cycle;
    \fill[gray,opacity=0.2] (x0.center)--(y0.center)--(x2.center)--cycle;
    \fill[gray,opacity=0.2] (y2.center)--(y0.center)--(x2.center)--cycle;
    \fill[gray,opacity=0.2] (x1.center)--(x2.center)--(y2.center)--cycle;
    \fill[gray,opacity=0.2] (x1.center)--(y1.center)--(y2.center)--cycle;
    
    \fill[blue,opacity=0.2] (x1.center)--(x2.center)--(x0.center)--cycle;
    \fill[blue,opacity=0.2] (y1.center)--(y2.center)--(y0.center)--cycle;
  \end{tikzpicture}
\end{figure}
\end{example}

We will observe in Example~\ref{ex:virtreg-noconverse} that $S/\sqrt{I}$ being virtually Cohen--Macaulay does not imply the same for $S/I$, even when $I$ is a monomial ideal. 
This example shows that the general monomial case cannot be reduced to the squarefree monomial case; moreover, it highlights that being virtually Cohen--Macaulay is a scheme-theoretic property rather than a set-theoretic one. 
For this reason, it is essential to develop tools to check the intuition developed through Theorem~\ref{thm:monomial-vcm} in the not-necessarily-radical case and to use the examples it delivers us to scaffold new ones as we build towards a theory of virtual depth. 
The remainder of this paper is directed at that transition.

\section{New virtual resolutions from old}
\label{sec:new-from-old}

We will now consider homological aspects of the virtually Cohen--Macaulay property.  In particular, we will now introduce two homological constructions that allow us to build new virtually Cohen--Macaulay modules from those we have shown to be virtually Cohen--Macaulay. Then, in the next section, we give homological obstructions to being virtually Cohen--Macaulay. 

For the remainder of the article, let $X$ be an arbitrary smooth projective toric variety with Cox ring $S$, and let $M$ be a finitely generated $\Pic(X)$-graded $S$-module. 

\subsection{A mapping cone construction}
\label{subsec:mapping-cone}

In this subsection, we introduce a mapping cone construction that, under certain conditions, will allow us to use a virtual resolution of the module $M$ to construct a shorter virtual resolution of $M$, given the right conditions. 

To begin, let $F_\bullet$ be a virtual resolution of $M$ of length $t$, and assume that $\Ext^t(M,S)^\sim=0$. Let $G^*_\bullet$ be a free resolution of $\Ext^t(M,S)$, shifted and with indexing reversed as in~\eqref{eqn:two-res}. 
By \cite[Prop. A3.13]{eisenbud}, there is an induced map, which we denote by $\alpha^*$, from $F_\bullet^* = \Hom^\bullet_S(F,S)$ to $G^*_\bullet$:
\begin{equation}
\label{eqn:two-res}
\begin{tikzcd}
\cdots\arrow[r]&0\arrow[r]\arrow[d,"\alpha_{0}^*"] &
F_0^*\arrow[r,"\varphi_1^*"]\arrow[d,"\alpha_{1}^*"] & 
F_1^*\arrow[r,"\varphi_2^*"] \arrow[d,"\alpha_{2}^*"] & 
\cdots \arrow[r,"\varphi_{t-2}^*"] & 
F_{t-2}^*\arrow[r,"\varphi_{t-1}^*"] \arrow[d,"\alpha_{t-2}^*"] & 
F_{t-1}^*\arrow[r,"\varphi_{t}^*"] \ar[d,"\alpha_{t-1}^*"] & 
F_t^*\arrow[r] \arrow[d,"\alpha_{t}^*"] &
0
\\
\cdots\arrow[r]&
G_{-1}^*\arrow[r,"\psi_{0}^*"'] &
G_0^*\arrow[r,"\psi_{1}^*"'] &
G_1^*\arrow[r,"\psi_{2}^*"'] &
\cdots\arrow[r,"\psi_{t-2}^*"']&
G_{t-2}^*\arrow[r,"\psi_{t-1}^*"'] & 
G_{t-1}^*\arrow[r,"\psi_{t}^* = \varphi_t^*"'] & 
G_t^*\arrow[r]& 
0.
\end{tikzcd}
\end{equation}
Dualizing yields the diagram
\begin{equation}
\label{eqn:tworesagain}
\begin{tikzcd}
\cdots&0\arrow[l]&F_0\arrow[l]  & 
F_1\arrow[l,"\varphi_1"']  & 
\cdots \arrow[l,"\varphi_{2}"']& 
F_{t-2}\arrow[l,"\varphi_{t-2}"'] & 
F_{t-1}\arrow[l,"\varphi_{t-1}"'] & 
F_t\arrow[l,"\varphi_{t}"'] &
\arrow[l] 0
\\
\cdots&
G_{-1}\arrow[l,"\psi_{-1}"] \arrow[u,"\alpha_{-1}"'] &
G_0\arrow[l,"\psi_{0}"] \arrow[u,"\alpha_{0}"'] &
G_1 \arrow[l,"\psi_{1}"] \arrow[u,"\alpha_{1}"'] & 
\cdots\arrow[l,"\psi_{2}"]&
G_{t-2}\arrow[l,"\psi_{t-2}"] \arrow[u,"\alpha_{t-2}"'] & 
G_{t-1}\arrow[l,"\psi_{t-1}"] \arrow[u,"\alpha_{t-1}"'] & 
G_t\arrow[l,"\psi_{t}=\varphi_t"] \arrow[u,"\alpha_{t}"'] & 
\arrow[l] 0.
\end{tikzcd}
\end{equation}
Then, the mapping cone of $\alpha\colon G\to F$, denoted $\cone(\alpha)$, is the complex
\begin{align*}
\cdots \gets G_{-2}\xleftarrow{\ \partial_0\ }
\begin{matrix} F_{0}\\ \oplus\\ G_{-1}\end{matrix}\xleftarrow{\ \partial_1\ } 
\begin{matrix} F_{1}\\ \oplus\\ G_{0}\end{matrix}\xleftarrow{\ \partial_2\ } 
\begin{matrix} F_{2}\\ \oplus\\ G_{1}\end{matrix}\xleftarrow{\ \partial_3\ } 
\cdots\gets 
\begin{matrix} F_{t-1}\\ \oplus\\ G_{t-2}\end{matrix}\xleftarrow{\ \partial_{t}\ } 
\begin{matrix} F_{t}\\ \oplus\\ G_{t-1}\end{matrix}\xleftarrow{\ \partial_{t+1}\ } 
\begin{matrix} 0\\ \oplus\\ G_{t}\end{matrix}\gets 0, 
\end{align*}
where the maps have the form 
$\partial_i = \begin{bmatrix} \varphi_i & \alpha_{i-1}\\ 0& \psi_{i-1}\end{bmatrix}$. 
Now, because $\psi_t = \varphi_t$, this reduces to the complex 
\begin{align}
\label{eqn:mapping-cone}
\cdots \gets G_{-2}\xleftarrow{\ \partial_0\ }
\begin{matrix} F_{0}\\ \oplus\\ G_{-1}\end{matrix}\xleftarrow{\ \partial_1\ } 
\begin{matrix} F_{1}\\ \oplus\\ G_{0}\end{matrix}\xleftarrow{\ \partial_2\ } 
\begin{matrix} F_{2}\\ \oplus\\ G_{1}\end{matrix}\xleftarrow{\ \partial_3\ } 
\cdots\gets 
\begin{matrix} F_{t-2}\\ \oplus\\ G_{t-3}\end{matrix}\gets 
G_{t-2}\gets 
0.
\end{align}

\begin{prop}
\label{prop:mapping-cone}
Let $S$ be the Cox ring of a smooth projective toric variety $X$, and let $M$ be a finitely generated $\Pic(X)$-graded $S$-module. 
Let $F_\bullet$ be a virtual resolution of $M$ of length $t$ such that $\Ext^t(M,S)^\sim = 0$, and let $\alpha$ be as in \eqref{eqn:tworesagain}. If $G_{-2}=0$ in \eqref{eqn:mapping-cone}, then (the minimization of) $\cone(\alpha)$ is a virtual resolution of $M$. 
\end{prop}
\begin{proof}
There is an exact triangle $G_\bullet \xrightarrow{\ \alpha\ }F_\bullet \to \cone(\alpha)\to G_\bullet[1]$, which induces the  long exact sequence in homology 
\[
\cdots\to H_{i+1}(\cone(\alpha))\to H_i(G)\to H_i(F)\to H_i(\cone(\alpha))\to\cdots.
\]
Since $H_i(G)^\sim = 0$ for all $i$, it follows that the homology modules of $\cone(\alpha)$ is isomorphic to those for $F_\bullet$, and thus $\cone(\alpha)$ and its minimization are virtual resolutions of $M$. 
\end{proof}

The mapping cone construction of  Proposition~\ref{prop:mapping-cone} can be iterated as long as the hypotheses hold. 

\begin{example}
\label{ex:disjoint-lines-continued}
Referring again to Example~\ref{ex:disjoint-lines}, the mapping cone construction of Proposition~\ref{prop:mapping-cone} also yields a short resolution for $S/J$.  
Since the variety $V(J)\subset X$ is simply the disjoint union of two lines, 
$S/J$ is not arithmetically Cohen--Macaulay even though it is Cohen--Macaulay at every relevant $\Pic(X)$-graded prime ideal. The minimal free resolution of $S/J$ is
\[
0\leftarrow S\leftarrow S(-2)^4\leftarrow S(-3)^4\leftarrow S(-4)\leftarrow 0.
\]
We will take the mapping cone of the following map of chain complexes, where the top chain complex is the dual of the free resolution of $\Ext^3(S/J,S) \cong k$:
\[
\begin{tikzcd}
0&S\arrow[l] \arrow[d,"\alpha_{-1}"] &
S(-1)^4\arrow[l,"\psi_{0}"'] \arrow[d,"\alpha_{0}"]&
S(-2)^6 \arrow[l,"\psi_{1}"'] \arrow[d,"\alpha_{1}"] & 
S(-3)^4\arrow[l,"\psi_{2}"'] \arrow[d,"\alpha_{2}"] & 
S(-4)\arrow[l,"\psi_{3}"'] \arrow[d,"\alpha_{3}"] & 
\arrow[l] 0.
\\
0 &
0\arrow[l]&S\arrow[l]  & 
S(-2)^4\arrow[l,"\varphi_1"']  &
S(-3)^4\arrow[l,"\varphi_2"']  &
S(-4) \arrow[l,"\varphi_3"']&
\arrow[l] 0.
\end{tikzcd}
\]
The mapping cone yields
\[
 S^2\leftarrow 
\begin{matrix} S(-1)^4\\ \oplus\\ S(-2)^4 \end{matrix}
\leftarrow 
\begin{matrix} S(-2)^6\\ \oplus\\  S(-3)^{4}\end{matrix}
\leftarrow 
\begin{matrix} S(-3)^4\\ \oplus\\ S(-4)^{\phantom{4}} \end{matrix}
\leftarrow S(-4)\leftarrow 0,
\]
which after minimizing provides a virtual resolution of $S/J$ of length $\codim(J)$:  
\[
S^2\leftarrow S(-1)^4\leftarrow S(-2)^2\leftarrow 0.
\]
Note that this resolution can also be constructed using the techniques of sheaves over simplicial complexes of \cite{yanagawa-sheaves}. 
\end{example}

\begin{example}
\label{ex:curve-mapping-cone}
Consider the hyperelliptic curve $C$ of genus $4$ which can be embedded as a curve of bidegree $(2,8)$ in $\PP^1 \times \PP^2$ found in \cite[Example~1.4]{virtual-original}, where $S = k[x_0,x_1,y_0,y_1,y_2]$. Now $S/I$ has minimal free resolution 
  \begin{align}
\nonumber
    &S^1 
    	\xleftarrow{\,\, \varphi_1\,\,}
      \begin{matrix} 
        S(-3,-1)^1 \\[-3pt]
        \oplus     \\[-3pt]
        S(-2,-2)^1 \\[-3pt]
        \oplus     \\[-3pt]
        S(-2,-3)^2 \\[-3pt]
        \oplus     \\[-3pt]
        S(-1,-5)^3 \\[-3pt]
        \oplus     \\[-3pt]
        S(0,-8)^1
      \end{matrix} 
	\xleftarrow{\,\, \varphi_2\,\,}
      \begin{matrix} 
        S(-3,-3)^3 \\[-3pt]
        \oplus     \\[-3pt] 
        S(-2,-5)^6 \\[-3pt] 
        \oplus     \\[-3pt] 
        S(-1,-7)^1 \\[-3pt] 
        \oplus     \\[-3pt]
        S(-1,-8)^2
      \end{matrix}
	\xleftarrow{\,\, \varphi_3\,\,}
      \begin{matrix} 
        S(-3,-5)^3 \\[-3pt] 
        \oplus     \\[-3pt] 
        S(-2,-7)^2 \\[-3pt] 
        \oplus     \\[-3pt] 
        S(-2,-8)^1
      \end{matrix}
	\xleftarrow{\,\, \varphi_4\,\,}
    	S(-3,-7)^1 \gets 0 \, .
\intertext{
Note that 
$\Ext^4_S(S/I,S)$ has finite length  
and that the dual of the shifted resolution of $\Ext^4_S(S/I,S)$ is 
}
     \begin{matrix} 
        S(-1, -1)^1
      \end{matrix}
	\xleftarrow{\,\, \psi_0\,\,}
     &\begin{matrix} 
        S(-1, -3)^3\\[-3pt] 
        \oplus     \\[-3pt]
 		S(-2, -1)^2
      \end{matrix}
	\xleftarrow{\,\, \psi_1\,\,}
     \begin{matrix} 
        S(-1, -5)^3\\[-3pt] 
        \oplus     \\[-3pt]
		S(-2, -3)^6\\[-3pt] 
        \oplus     \\[-3pt]
		S(-3, -1)^1
      \end{matrix}
	\xleftarrow{\,\, \psi_2\,\,}
     \begin{matrix} 
        S(-2, -8) \\[-3pt] 
        \oplus     \\[-3pt]
		S(-1, -7)\\[-3pt] 
        \oplus     \\[-3pt]
		S(-2, -5)^6 \\[-3pt] 
        \oplus     \\[-3pt]
		S(-3, -3)^3
      \end{matrix}
	\xleftarrow{\,\, \psi_3\,\,}
     \begin{matrix} 
        S(-3, -5)^3 \\[-3pt] 
        \oplus     \\[-3pt]
		S(-2, -7)^2\\[-3pt] 
        \oplus     \\[-3pt]
		S(-2, -8)^1
      \end{matrix}
	\xleftarrow{\,\, \psi_4\,\,}
    	S(-3, -7)^1
\gets
	0\, .
\intertext{Applying Proposition~\ref{prop:mapping-cone} yields a virtual resolution for $S/I$ of the form }
\label{eqn:mapcone1}
&\begin{matrix} 
        S^1 \\[-3pt]
        \oplus     \\[-3pt]
        S(-1,-1)^1
      \end{matrix}
	\xleftarrow{\,\, \rho_1\,\,}
     \begin{matrix} 
        S(-2,-1)^1 \\[-3pt]
        \oplus     \\[-3pt]
        S(-2,-2)^1 \\[-3pt]
        \oplus     \\[-3pt]
        S(-1,-3)^3 \\[-3pt]
        \oplus     \\[-3pt]
        S(-0,-8)^1 \\[-3pt]
        \oplus     \\[-3pt]
        S(-2,-1)^1
      \end{matrix} 
	\xleftarrow{\,\, \rho_2\,\,}
     \begin{matrix} 
        S(-2,-3)^3 \\[-3pt]
        \oplus     \\[-3pt]
        S(-1,-8)^2 \\[-3pt]
        \oplus     \\[-3pt]
        S(-2,-3)^1
      \end{matrix} 
	\xleftarrow{\,\, \rho_3\,\,}
 S(-2,-8)^{1} \gets 
0.
\intertext{
However, the cokernel of $\rho_3^*$ is also irrelevant, so this procedure can be repeated in this case, so applying Proposition~\ref{prop:mapping-cone} to \eqref{eqn:mapcone1} yields the following virtual resolution for $S/I$:
}
\nonumber
     &\begin{matrix} 
        S^1 \\[-3pt]
        \oplus     \\[-3pt]
        S(0,-1)^2 \\[-3pt]
        \oplus     \\[-3pt]
        S(0,-2)^1 \\[-3pt]
        \oplus     \\[-3pt]
        S(-0,-2)^1
      \end{matrix} 
	\xleftarrow{\,\, \partial_1\,\,}
     \begin{matrix} 
        S(-1,-1)^2 \\[-3pt]
        \oplus     \\[-3pt]
        S(-1,-2)^1 \\[-3pt]
        \oplus     \\[-3pt]
        S(0,-3)^1 \\[-3pt]
        \oplus     \\[-3pt]
        S(-1,-2)^1 \\[-3pt]
        \oplus     \\[-3pt]
        S(0,-3)^1 \\[-3pt]
        \oplus		\\[-3pt]
        S(-1,-1)^1 \\[-3pt]
        \oplus		\\[-3pt]
        S(0,-3)^2
      \end{matrix} 
	\xleftarrow{\,\, \partial_2\,\,}
	S(-1,-3)^{5}\gets 0, 
\end{align}
where 
\[
\partial_1 = \begin{bmatrix}
    0&-{x}_{1}{y}_{0}-{x}_{1}{y}_{1}&{x}_{0}{y}_{0}^{2}&{y}_{1}^{2}{y}_{2}&-{x}_{1}{y}_{1}^{2}-{x}_{0}{y}_{2}^{2}&{y}_{0}{y}_{2}^{2}+{y}_{1}{y}_{2}^{2}&{x}_{0}{y}_{2}&{y}_{0}^{3}+{y}_{0}^{2}{y}_{1}&0\\
    {-{x}_{1}}&{-{x}_{0}}&0&{y}_{0}^{2}&0&{-{y}_{1}^{2}}&0&{-{y}_{2}^{2}}&0\\
    0&0&{-{x}_{1}}&0&{-{x}_{0}}&{y}_{0}+{y}_{1}&0&0&{-{y}_{2}}\\
    {x}_{0}&0&0&{y}_{2}^{2}&0&0&{-{x}_{1}}&{-{y}_{1}^{2}}&{y}_{0}^{2}
    \end{bmatrix}
\]
and 
\[
\partial_2 = \begin{bmatrix}
    {-{y}_{1}^{2}}&0&{y}_{0}^{2}&{-{y}_{2}^{2}}&0\\
    {y}_{2}^{2}&0&0&{y}_{0}^{2}&{-{y}_{1}^{2}}\\
    {y}_{0}+{y}_{1}&{-{y}_{2}}&0&0&0\\
    0&0&{x}_{1}&{x}_{0}&0\\
    0&0&{y}_{2}&0&{y}_{0}+{y}_{1}\\
    {x}_{1}&0&0&0&{x}_{0}\\
    0&{y}_{0}^{2}&{y}_{2}^{2}&{-{y}_{1}^{2}}&0\\
    {-{x}_{0}}&0&0&{x}_{1}&0\\
    0&{x}_{1}&{-{x}_{0}}&0&0
\end{bmatrix}.
\]
\end{example}

\subsection{The quotient by a virtually regular element}
\label{subsec:vreg-element}

The purpose of this subsection is to introduce the notion of a virtually regular element and to show that the quotient of a virtually Cohen--Macaulay module by a virtually regular element is again virtually Cohen--Macaulay.  We do this by the explicit construction of a virtual resolution of the appropriate length for the quotient module arising from a virtual resolution of the original module. Recall that a module $M$ is irrelevant if $\widetilde{M} = 0$.

\begin{definition}
\label{def:vreg-element}
Let $f \in S$ be homogeneous, and let $M$ be an $S$-module.   If $\Ann_M f$ is irrelevant and $\dim M = 1+\dim M/fM$, then we say that $f$ is \emph{virtually regular on $M$} or that $f$ is a \emph{virtually regular element on $M$}.
\end{definition}

It is immediate that any regular element on $M$ is virtually regular and that no element of a minimal prime of $M$ can be virtually regular.  The additional flexibility gained in considering virtually regular elements over regular elements alone is that an element of an embedded associated prime of $M$ can be virtually regular if its annihilator is sufficiently well controlled.  Notice also that if $M'$ is an $S$-module satisfying $\widetilde{M'} = \widetilde{M}$, then $f$ is virtually regular on $M$ if and only if $f$ is virtually regular on $M'$.   We will see below that quotienting by a virtually regular element preserves the virtually Cohen--Macaulay property, just as quotienting by a regular element preserves the Cohen--Macaulay property of modules.  However, unlike in the affine setting, the converse is not true (see Example \ref{ex:virtreg-noconverse}).

\begin{prop}
\label{prop:vreg-elt-quotient}
Let $S$ be the Cox ring of a smooth projective toric variety $X$, and let $M$ be a finitely generated $\Pic(X)$-graded $S$-module. 
If $M$ has a virtual resolution of length $\ell$ and $f$ is a virtually regular element on $M$, then $M/fM$ has a virtual resolution of length $\ell+1$.  In particular, if $M$ is virtually Cohen--Macaulay, then $M/fM$ is either virtually Cohen--Macaulay or irrelevant.
\end{prop}
\begin{proof}
Because $\dim M=1+\dim M/fM$, it suffices to prove the first claim.  Let $F_\bullet$ be a virtual resolution of $M$ of length $\ell$.  Consider the complex $G_\bullet = 0 \rightarrow S \xrightarrow{f} S \rightarrow S/\<f\> \rightarrow 0$.  We claim that the total complex, $E_\bullet$, of the double complex of $F_\bullet\otimes_S G_\bullet$ 
gives a virtual resolution of $M/fM$.  It is clear that $H_0(E) = M' \otimes_S S/\<f\>$ for some module $M'$ satisfying $\widetilde{M'} = \widetilde{M}$. 
Because $\Ann_M f$ is irrelevant, it follows that $H_0(E)^\sim = (M/fM)^\sim$. 
A standard diagram chase shows that the higher homology of the total complex is irrelevant.  
Because $E_\bullet$ has length $\ell+1$, we have found a virtual resolution of $M/fM$ of length $\ell+1$, as desired.
\end{proof}

\begin{definition}
We say that the sequence $f_1, \ldots, f_k$ is a \emph{virtually regular sequence} on the module $M$ if $f_1$ is virtually regular on $M$ and if $f_i$ is virtually regular on $M/ \langle f_1, \ldots, f_{i-1} \rangle M$ for all $1<i \le k$.
\end{definition}

The following corollary is a immediate from Proposition \ref{prop:vreg-elt-quotient}. 

\begin{corollary}
If the $S$-module $M$ is virtually Cohen--Macaulay, and $f_1, \ldots, f_k$ is a virtually regular sequence on $M$, then $M/ \langle f_1, \ldots, f_k \rangle M$ is either virtually Cohen--Macaulay or irrelevant.
\end{corollary}

\begin{example}
\label{ex:vreg-element}
Let $S = k[x_0,\ldots,x_5]$ be the Cox ring of $\PP^5$ 
and consider the ideal 
\[
J = \<x_0, x_1,x_2\> \cap \<x_3,x_4,x_5\>.
\]  
With $M = S/J$ and $F_\bullet$ the minimal free resolution of $M$, the construction in Proposition~\ref{prop:mapping-cone} yields a virtual resolution of $M$ of length $\codim M = 3$, which shows that $M$ is virtually Cohen--Macaulay.  
We claim that $x_2-x_5$ is a virtually regular element on $M$, that $x_1-x_4$ is a virtually regular element on $M/\langle x_2-x_5 \rangle M$, and that $x_0-x_3$ is a virtually regular element on $M/ \langle x_2-x_5, x_1-x_4 \rangle M$.  
Because $x_2-x_5$ is a regular element, it is automatically a virtually regular element.  Observe that 
\begin{align*}
\overline{M}=\frac{M}{\langle x_2-x_5 \rangle M} &\cong \frac{S}{ \langle x_0x_3,x_0x_4,x_0x_2,x_1x_3,x_1x_4,x_1x_2,x_2x_3,x_2x_4,x_2^2,x_5 \rangle}\\
& \cong  \frac{S}{\langle x_0,x_1,x_2,x_5 \rangle \cap \langle x_2,x_3,x_4,x_5 \rangle \cap \langle x_0,x_1,x_2^2,x_3,x_4,x_5 \rangle }.
\end{align*} 
Now $x_1-x_4$ is not a regular element on $\overline{M}$, but it is not in either minimal prime of $\overline{M}$, and so $\dim \overline{M} = 1+\dim \overline{M}/ \langle x_1-x_4 \rangle \overline{M}$.  
The isomorphism presented above is given by $x_i \mapsto x_i$ for $i \neq 5$ and $x_5 \mapsto x_2-x_5$.  After application of this isomorphism, it is easy to see that $\Ann_{\overline{M}} \langle x_1-x_4 \rangle = \langle x_2 \rangle \overline{M}$, which is irrelevant.  Hence, $x_1-x_4$ is a virtually regular element that is not a regular element on $\overline{M}$.  A similar computation shows that $x_0-x_3$ is in an embedded prime of $M/\langle x_2-x_5, x_1-x_4 \rangle M$ and has, after applying an analogous isomorphism to the one described above, an irrelevant annihilator generated by $x_1$ and $x_2$.  Hence, $x_2-x_5, x_1-x_4,x_0-x_3$ is a virtually regular sequence on $M$.

Therefore, since $M$ is virtually Cohen--Macaulay, so are each of the modules $M/\langle x_2-x_5 \rangle M$ and $M/ \langle x_2-x_5, x_1-x_4 \rangle M$ by Proposition~\ref{prop:vreg-elt-quotient}.  Notice that $M/ \langle x_2-x_5, x_1-x_4,x_0-x_3 \rangle M$ is irrelevant.
\end{example}

It is worth noting, however, that the converse to Proposition~\ref{prop:vreg-elt-quotient} is false, even over the Cox ring of a single projective space, as seen in the example below.

\begin{example}\label{ex:virtreg-noconverse}
If $S = k[x_0,x_1,x_2]$ is the Cox ring of $\PP^2$, and $M = S/\langle x_0^2,x_0x_1 \rangle$, then there is no $f \in S$ so that $S/\< f \>^\sim \cong M^\sim$. Thus, the virtual dimension of $M$ is at least of length $2$ while $\codim M=1$, so $M$ is not virtually Cohen--Macaulay. 
However, $x_2$ is a (virtually) regular element on $M$, and $M/\<x_2\>M^\sim \cong S/\<x_0,x_2\>^\sim$, which is clearly virtually Cohen--Macaulay.  
Hence, we have an example of a module that is not virtually Cohen--Macaulay and a virtually regular on it so that the quotient by that virtually regular element yields virtually Cohen--Macaulay module.  Notice also that $x_2, x_1,x_0$ is a virtually sequence on $M$ while $x_0,x_1,x_2$ is not, which shows that virtually regular sequences need not be permutable. Additionally, because $S/\langle x_0 \rangle$ is virtually Cohen--Macaulay, this example shows that it is possible that a monomial ideal $I$ does not define a virtually Cohen--Macaulay scheme while its radical $\sqrt{I}$ does.  
\end{example}

\section{Derived functors via virtual resolutions}
\label{sec:virtual-ext}

In this section, we record that the sheafifications of $\Ext$ and $\Tor$ functors can be computed using virtual resolutions in place of free resolutions.  This perspective gives necessary conditions on a module for it to be virtually Cohen--Macaulay. 
The goals of this treatment are to record the relationship between the length of shortest virtual resolutions with homological dimension and to  note some conditions on modules that prevent them from being virtually Cohen--Macaulay. These conditions are those that would prevent a module from having any virtual resolution of length equal to its codimension.

As in the previous section, $X$ will always denote an arbitrary smooth projective toric variety with Cox ring $S$ and that all $S$-modules will be finitely generated and $\Pic(X)$-graded.  Nevertheless, we will endeavor to repeat these hypotheses in the statements of theorems.

\begin{prop}\label{prop:vExt}
Let $M$ and $N$ be finitely generated $\Pic(X)$-graded modules over the Cox ring $S$ of a smooth projective toric variety $X$. 
If $F_\bullet$ is any virtual resolution of $M$, then $\Ext^i_S(M,N)^\sim$ is the sheafification of the $i^{th}$ homology module of $Hom_S(F_\bullet,N)$.  
\end{prop}

\begin{proof}
Fix a virtual resolution $F_\bullet$ of $M$ and an injective resolution $I^\bullet$ of $N$.  We consider the spectral sequence of the double complex whose $(i,j)^{th}$ entry on page $0$ is $\Hom_S(F_i, I^j)$.  Taking homology along each column yields a page 1 whose only nonzero entries are $\Hom_S(F_\bullet,N)$.  Hence, the module in position $(i,j)$ on page $\infty$ is the $i^{th}$ homology module of $\Hom_S(F_\bullet,N)$.  

 As usual, let $B$ denote the irrelevant ideal of $S$. On the other hand, taking homology along rows first yields a page 1 whose entry in position $(i,j)$ is $\Hom(L_i, I^j)$ where $L_i$ is the $i^{th}$ homology of $F_\bullet$, which is supported only on $B$ unless $i = 0$, in which case $L_0 = M$.  Hence, the entry in position $(i,j)$ on page $\infty$ is $H_i \oplus K_i$ for some module $K_i$ supported only on $B$ and $H_i \subseteq \Ext^i_S(M,N)$ with $\Ext^i_S(M,N)/H_i$ supported only on $B$.  Hence, $\Ext^i_S(M,N)^\sim \cong (H_i \oplus K_i)^\sim$, as desired.
\end{proof}

It is a classical result that, for a module $M$ over the Cox ring $S$ of $\PP^d$, the condition that $H^i_B(M)^\vee \cong \Ext_S^{d+1-i}(M,S)$ be irrelevant (equivalently, finite length) for all $i<\dim(M)$ is equivalent to the condition that $M$ be equidimensional and Cohen--Macaulay at its localization at every relevant prime in its support.  
We will see in Section~\ref{sec:VCMvsOthers} that virtually Cohen--Macaulay implies geometrically Cohen--Macaulay and equidimensional. 
Putting these facts together, we have that when $X$ is a single projective space,  $M$ being virtually Cohen--Macaulay implies that $\Ext_S^j(M,S)$ is irrelevant for all $j>\codim(M)$. 
The following corollary, which is immediate from Proposition~\ref{prop:vExt}, extends that result to the arbitrary smooth projective toric setting.  

\begin{corollary}
\label{cor:vanishingExt}
Let $S$ be the Cox ring of a smooth projective toric variety $X$, and let $M$ be a finitely generated $\Pic(X)$-graded $S$-module. 
If $M$ has a virtual resolution of length $\ell$, then $\Ext^i_S(M,N)^\sim = 0$ for all $\Pic(X)$-graded finitely generated $S$-modules $N$ and all $i>\ell$. \qed
\end{corollary}

\begin{example}\label{ex:disjoint-lines-P4}
In Example~\ref{ex:disjoint-lines}, we saw that $I = \<x_0,x_1\> \cap \<x_2,x_3\>$ defined a virtually Cohen--Macaulay subscheme of $\PP^3$.  
Corollary~\ref{cor:vanishingExt} implies that $I$ does not define a virtually Cohen--Macaulay subscheme of $\PP^d$ whenever $d>3$. 
In particular, with $S = k[x_0,\ldots,x_d]$, we have that $\Ext_S^3(S/I,S) \cong S/\<x_0,x_1,x_2,x_3\>$, which is not irrelevant.
\end{example}

\begin{remark}
With notation as above, the virtual dimension of $M$ is greater than or equal to the homological dimension of $\widetilde{M}$.  To see this, observe that if $M$ has a virtual resolution of length $\ell$, then, because every virtual resolution of $M$ gives rise to a locally free resolution of $\widetilde{M}$, the homological dimension of $\widetilde{M}$ is at most $\ell$.  However, it is not true that a module $M$ is virtually Cohen--Macaulay if and only if its sheafification has homological dimension equal to its codimension.  For example, any module whose sheafification is a vector bundle that does not split as a direct sum of line bundles has homological dimension $0$ while its shortest virtual resolution must have positive length. 
Example~\ref{ex:tangentbundle} examines an explicit example of this type.  
\end{remark}

We have seen in Corollary~\ref{cor:vanishingExt} that an $S$-module $M$ cannot have a virtual resolution of length $\ell$ if there is some $\Ext^i_S(M,N)^\sim \neq 0$ for some $S$-module $N$ and some $i>\ell$.  
The following two corollaries combine to show that one need not consider all possible $N$ but that it is instead sufficient to check only $N=S$.

\begin{corollary}\label{cor:inductiveStep}
Let $S$ be the Cox ring of a smooth projective toric variety $X$, and let $M$ be a finitely generated $\Pic(X)$-graded $S$-module. If $\Ext_S^\ell(M,S)^\sim= 0$ and $\Ext_S^{\ell+1}(M,L)^\sim = 0$ for every finitely generated $\Pic(X)$-graded $S$-module $L$, then $\Ext_S^\ell(M,N)^\sim = 0$ for every finitely generated $\Pic(X)$-graded $S$-module $N$.
\end{corollary}

\begin{proof}
Let $N$ be an $S$-module. 
There is a short exact sequence (of $S$-modules but typically not of graded $S$-modules) for some $a \ge 1$ and some $K$ of the form 
$0 \longrightarrow K \longrightarrow S^a \longrightarrow N \longrightarrow 0$.
Applying $\Ext_S(M,-)$ to this yields
\[ 
\Ext_S^\ell(M,S^a)^\sim \longrightarrow \Ext_S^\ell(M,N)^\sim \longrightarrow
\Ext_S^{\ell+1}(M,K)^\sim = 0.
\]
Since $\Ext_S^{\ell}(M, S^a) \cong (\Ext_S^{\ell}(M,S))^a= 0$, 
it must be true that $\Ext_S^\ell(M,N)^\sim= 0$.
\end{proof}

\begin{cor}
Let $S$ be the Cox ring of a smooth projective toric variety $X$. 
Suppose the finitely generated $\Pic(X)$-graded $S$-module $M$ has the property that $\Ext_S^i(M,S)$ is irrelevant for all $i \ge \ell$, for some $\ell \ge 0$. 
Then $\Ext_S^i(M,N)$ is irrelevant for every finitely generated $\Pic(X)$-graded $S$-module $N$ for all $i \ge \ell$.
\end{cor}

\begin{proof}
Because free resolutions are virtual resolutions, the claim is trivial if $\ell$ is greater than the projective dimension of $M$, denoted $\pdim M$, so we assume that $\ell \le \pdim M$.   Now, using Corollary~\ref{cor:inductiveStep}, we proceed by induction on $\pdim M-\ell$.
\end{proof}

For completeness, we state an analogue for $\Tor$ of Proposition~\ref{prop:vExt}, which concerned $\Ext$.

\begin{prop}
\label{prop:vTor}
Let $M$ and $N$ be finitely generated $\Pic(X)$-graded modules over the Cox ring $S$ of a smooth projective toric variety $X$.  If  $F_\bullet$ is any virtual resolution of $M$, then $\Tor^S_i(M,N)^\sim$ is the sheafification of the $i^{th}$ homology module of $F_\bullet \otimes_S N$.
\end{prop}

\begin{proof}
The argument follows the proof of  Proposition~\ref{prop:vExt} but uses the spectral sequence arising from the double complex $F_\bullet \otimes_S G_\bullet$, where $G_\bullet$ is a free resolution of $N$.
\end{proof}

\begin{corollary}
\label{cor:vanishingTor}
Let $S$ be the Cox ring of a smooth projective toric variety $X$. 
If a finitely generated $\Pic(X)$-graded $S$-module $M$ has a virtual resolution of length $\ell$, then $\Tor_i^S(M,N)^\sim = 0$ for all finitely generated $\Pic(X)$-graded $S$-modules $N$ and all $i>\ell$. \qed
\end{corollary}

Connecting back to Definition~\ref{def:vreg-element}, a virtually regular element has a description in terms of virtual $\Tor$, just as, in the affine case, a regular element can be described in terms of the vanishing of certain $\Tor$ modules.

\begin{prop}
Let $S$ be the Cox ring of a smooth projective toric variety $X$, 
$M$ be a finitely generated $\Pic(X)$-graded $S$-module, and $f\in S$ be homogeneous. 
Then $f$ is virtually regular on $M$ (as in Definition~\ref{def:vreg-element}) if and only if $1+\dim M/fM = \dim M$ and 
$\Tor_1^S(M,S/\<f\>)^\sim = 0$.  
\end{prop}
\begin{proof}
Tensor the short exact sequence 
$0 \rightarrow S \xrightarrow{f} S \rightarrow S/\<f\> \rightarrow 0$  
with $M$ to see that there is an isomorphism of $S$-modules $\Tor^S_1(M,S/\<f\>) \cong \Ann_M f$.
\end{proof}

\section{Relationships among the arithmetically, virtually, and geometrically Cohen--Macaulay properties}
\label{sec:VCMvsOthers}

As in the previous section, $X$ will always denote an arbitrary smooth projective toric variety with Cox ring $S$ and irrelevant ideal $B$ and that all $S$-modules will be finitely generated and $\Pic(X)$-graded.  We begin this section by recording a relationship between the arithmetically Cohen--Macaulay, virtually Cohen--Macaulay, and geometrically Cohen--Macaulay properties.  Recall that an ideal $I$ is $\emph{relevant}$ if $B^t \not\subseteq I$ for all $t \geq 1$.

\begin{defn}\label{def:geomCM}
An $S$-module $M$ is \emph{geometrically Cohen--Macaulay} if $M_P$ is a Cohen--Macaulay $S_P$-module for all relevant primes $P$ in the support of $M$.
\end{defn}

 Note that the condition that $M$ be geometrically Cohen--Macaulay is equivalent to the condition that $\widetilde{M}$ be a Cohen--Macaulay sheaf on $X$.

\begin{prop}\label{prop:aCM=>vCM=>gCM}
Let $S$ be the Cox ring of a smooth projective toric variety $X$.  If $M$ is a finitely generated $\Pic(X)$-graded $S$-module, then \begin{enumerate}
\vspace{-2mm}
\item if $M$ is arithmetically Cohen--Macaulay, then $M$ is virtually Cohen--Macaulay; and 
\item if $M$ is virtually Cohen--Macaulay, then $M$ is geometrically Cohen--Macaulay.
\end{enumerate}
\end{prop}
\begin{proof}
If $M$ is arithmetically Cohen--Macaulay, then by the Auslander--Buchsbaum formula, $M$ has a free resolution of length $\codim M$. 
Because free resolutions are virtual resolutions, $M$ is thus virtually Cohen--Macaulay.  

Similarly, if $M$ is virtually Cohen--Macaulay, then it has a virtual resolution $F_\bullet$ of length $\codim M$. 
If $P$ is any relevant prime in the support of $M$, then localizing $F_\bullet$ at $P$ gives a free resolution $(F_P)_\bullet$ of $M_P$ of length $\codim M$.  Because the codimension of $\Spec(S/\Ann_S(M))$ in $\Spec(S)$ is equal to the codimension of $\Spec(S_P/\Ann_{S_P}(M_P))$ in $\Spec(S_P)$, it follows from the Auslander--Buchsbaum formula that $M_P$ is Cohen--Macaulay. 
Hence $M$ is geometrically Cohen--Macaulay.
\end{proof}

We saw several times in Section~\ref{sec:triangles}, for example in Example~\ref{ex:disjoint-lines}, that implication (1) of Proposition~\ref{prop:aCM=>vCM=>gCM} is strict.  
We now give an example showing that implication (2) is also strict.

\begin{example}
\label{ex:tangentbundle}
Even over the Cox ring of projective space, a module can be geometrically Cohen--Macaulay but not virtually Cohen--Macaulay. 
For example, if $S$ is the Cox ring of $\PP^d$ with $d > 1$ and $M$ corresponds to the tangent bundle, i.e., $M$ is the cokernel of the map 
\vspace{-1mm}
\[
S^{d+1} \xleftarrow{ 
	\begin{bmatrix}
      x_0 \\
      \vdots \\
      x_d
      \end{bmatrix}
      } S \leftarrow 0, 
\vspace{-1mm}
\] 
then $\widetilde{M}$ is a vector bundle that does not split as a direct sum of line bundles, see~\cite[Theorem~8.1.6]{cox-little-schenck}. 
Thus $M$ has virtual dimension $1$ but codimension $0$. 

Meanwhile, for each $0 \le i \le d$, the matrix above has a unit entry after tensoring with $S[1/x_i]$, which shows that $M[1/x_i] \cong S[1/x_i]$, and so $M$ is a geometrically Cohen--Macaulay.  In fact, not only is $M$ geometrically Cohen--Macaulay, but also it is a faithful module of depth $d$ on the homogeneous maximal ideal of $S$. 
These properties show that the virtual Cohen--Macaulay property is not captured by the geometric Cohen--Macaulay property along with depth information coming from the affine setting.  
\end{example}

It was shown in \cite[Proposition 5.1]{virtual-original} that every $B$-saturated virtually Cohen--Macaulay module over the Cox ring $S$ of a product of projective spaces is unmixed.  The argument, which we record here, generalizes to arbitrary smooth projective toric varieties. 

\begin{prop}
\label{prop:equidim}
Let $S$ be the Cox ring of a smooth projective toric variety $X$ with irrelevant ideal $B$. 
If $M$ is a finitely generated $\Pic(X)$-graded $B$-saturated $S$-module that is virtually Cohen--Macaulay, then $\dim S/P = \dim M$ for all associated primes $P$ of $M$. 
\end{prop}
\begin{proof}
Suppose that $M$ is a virtually Cohen--Macaulay module of codimension $c$, and suppose that there is some associated prime $P$ of $M$ of codimension $e>c$. 
Let $F_{\bullet}$ be a virtual resolution of length $c$. 
Because $M$ is $B$-saturated, every associated prime $P$ of $M$ is relevant.  
Hence, $(F_P)_\bullet$ gives an $S_P$-free resolution of $M_P$ of length $c$. 
Let $\pdim_{S_P} M_P$ denote the projective dimension of $M_P$ over $S_P$, and recall that the codimension of $\Spec(S/\Ann_S(M))$ in $\Spec(S)$ is equal to the codimension of $\Spec(S_P/\Ann_{S_P}(M_P))$ in $\Spec(S_P)$, which we record as $\codim M = \codim M_P$.  Then we obtain a contradiction because 
\[
\pdim M_P \le c<e = \codim M = \codim M_P. \qedhere
\] 
\end{proof}

Proposition~\ref{prop:equidim} motivates the definition of a \emph{virtually unmixed $S$-module}: 

\begin{defn}
Let $M$ be an $S$-module, and let $N$ denote the $B$-saturated $S$-module satisfying $\widetilde{N} = \widetilde{M}$.  We say that $M$ is \emph{virtually unmixed} if it satisfies the following conditions, which are easily seen to be equivalent.
\begin{enumerate}
\item For all relevant associated primes $P$ of $M$, $\dim S/P = \dim M$.
\item For all associated primes $P$ of $N$, $\dim S/P = \dim N$.
\item The $B$-saturation of $\Ann_S M$ is an unmixed ideal.
\item The annihilator $\Ann_S N$ is an unmixed ideal.
\end{enumerate}
\end{defn}

The following corollary is immediate from Proposition~\ref{prop:equidim}.

\begin{corollary}
\label{cor:irrelevantEmbedded}
If $M$ is a virtually Cohen--Macaulay $S$-module, then $M$ is virtually unmixed.
\end{corollary}

\raggedbottom

\begin{bibdiv}
\begin{biblist}

\bib{audin}{book}{
    AUTHOR = {Audin, Mich\`ele},
     TITLE = {The topology of torus actions on symplectic manifolds},
    SERIES = {Progress in Mathematics},
    VOLUME = {93},
      NOTE = {Translated from the French by the author},
 PUBLISHER = {Birkh\"{a}user Verlag, Basel},
      YEAR = {1991},
     PAGES = {181},
}

\bib{virtual-original}{article}{
    AUTHOR = {Berkesch, Christine},
    AUTHOR = {Erman, Daniel},
    AUTHOR = {Smith, Gregory G.},
     TITLE = {Virtual resolutions for a product of projective spaces},
   JOURNAL = {Alg. Geom.},
  FJOURNAL = {Algebraic Geometry},
    VOLUME = {7},
      YEAR = {2020},
    NUMBER = {4},
     PAGES = {460--481},
}

\bib{cox:homog}{article}{
    AUTHOR = {Cox, David A.},
     TITLE = {The homogeneous coordinate ring of a toric variety},
   JOURNAL = {J. Algebraic Geom.},
  FJOURNAL = {Journal of Algebraic Geometry},
    VOLUME = {4},
      YEAR = {1995},
    NUMBER = {1},
     PAGES = {17--50},
}

\bib{cox-little-schenck}{book}{
    AUTHOR = {Cox, David A.},
    AUTHOR = {Little, John B.},
    AUTHOR = {Schenck, Henry K.},
     TITLE = {Toric varieties},
    SERIES = {Graduate Studies in Mathematics},
    VOLUME = {124},
 PUBLISHER = {American Mathematical Society, Providence, RI},
      YEAR = {2011},
     PAGES = {xxiv+841},
}

\bib{eisenbud}{book}{
    AUTHOR = {Eisenbud, David},
     TITLE = {Commutative algebra with a view toward algebraic geometry},
    SERIES = {Graduate Texts in Mathematics},
    VOLUME = {150},
 PUBLISHER = {Springer-Verlag, New York},
      YEAR = {1995},
     PAGES = {xvi+785},
}

\bib{Har}{book}{
    AUTHOR = {Hartshorne, Robin},
     TITLE = {Algebraic Geometry},
    SERIES = {Graduate Texts in Mathematics},
    VOLUME = {52},
 PUBLISHER = {Springer-Verlag, New York},
      YEAR = {2000},
      }


\bib{monomial-radical}{article}{
   author={Herzog, J.},
   author={Takayama, Y.},
   author={Terai, N.},
   title={On the radical of a monomial ideal},
   journal={Arch. Math. (Basel)},
   volume={85},
   date={2005},
   number={5},
   pages={397--408},
}

\bib{reu2019}{article}{
   author={Kenshur, Nathan},
   author={Lin, Feiyang}, 
   author={McNally, Sean}, 
   author={Xu, Zixuan}, 
   author={Yu, Teresa}, 
   title={On virtually Cohen--Macaulay simplicial complexes},
   journal={arXiv:2007.09443},
}

\bib{Lop}{article}{
  author = {Loper, Michael C.},
  title = {What makes a complex virtual},
  journal = {arXiv:1904.05994},
  }
  
  \bib{MS04}{book}{
    author = {Miller, Ezra},
    author =  {Sturmfels, Bernd},
    Title = {Combinatorial Commutative Algebra},
    SERIES = {Graduate Texts in Mathematics},
    YEAR = {2004},
    VOLUME = {227},
	PUBLISHER = {Springer-Verlag, New York},
    YEAR = {2005},
    PAGES = {xiv+417},
    }

\bib{musson}{article}{
   author={Musson, Ian M.},
   title={Differential operators on toric varieties},
   journal={J. Pure Appl. Algebra},
   volume={95},
   date={1994},
   number={3},
   pages={303--315},
}

\bib{mustata-toric}{article}{
   author={Musta\c{t}\u{a}, Mircea},
   title={Vanishing theorems on toric varieties},
   journal={Tohoku Math. J. (2)},
   volume={54},
   date={2002},
   number={3},
   pages={451--470},
}

\bib{yanagawa-sheaves}{article}{
   author={Yanagawa, Kohji},
   title={Stanley--Reisner rings, sheaves, and Poincar\'{e}--Verdier duality},
   journal={Math. Res. Lett.},
   volume={10},
   date={2003},
   number={5-6},
   pages={635--650},
}

\bib{yang-monomial}{article}{
	author={Yang, Jay}, 
	title={Virtual resolutions of monomial ideals on toric varieties}, 
	journal={Proc. Amer. Math. Soc. Ser. B, to appear},
	year={2021}
}

\end{biblist}
\end{bibdiv}
\def\cprime{$'$} \def\cprime{$'$}
\providecommand{\MR}{\relax\ifhmode\unskip\space\fi MR }
\providecommand{\MRhref}[2]{%
  \href{http://www.ams.org/mathscinet-getitem?mr=#1}{#2}
}
\providecommand{\href}[2]{#2}

\end{document}